\documentclass[11pt]{amsart}

\usepackage{amsfonts, amstext, amsmath, amsthm, amscd, amssymb}
\usepackage{graphicx, color, epsfig, subfigure, wrapfig, overpic}
\usepackage[all,cmtip]{xy} 

\AtBeginDocument{%
   \def\MR#1{}
}

\newcommand{\R}{\mathbb{R}}

\newcommand{\Z}{\mathbb{Z}}

\newcommand{\CC}{\mathbb{C}}
\newcommand{\RR}{\mathbb{R}}
\newcommand{\HH}{\mathbb{H}}
\newcommand{\ZZ}{\mathbb{Z}}

\renewcommand{\P}{\mathcal P}

\newcommand{\W}{\mathcal W}
\newcommand{\vol}{{\rm vol}}
\newcommand{\cut}{{\backslash \backslash}}
\newcommand{\guts}{{\rm guts}}

\newcommand{\bdy}{\partial}

\newcommand{\specvol}{{\rm Spec}_{\vol}}
\newcommand{\specdet}{{\rm Spec}_{\rm det}}

\newcommand{\voct}{{v_{\rm oct}}}
\newcommand{\vtet}{{v_{\rm tet}}}

\theoremstyle{plain}
\newtheorem{theorem}{Theorem}[section]
\newtheorem{corollary}[theorem]{Corollary}
\newtheorem{lemma}[theorem]{Lemma}
\newtheorem{prop}[theorem]{Proposition}

\newtheorem{conjecture}[theorem]{Conjecture}

\newtheorem*{namedtheorem}{\theoremname}
\newcommand{\theoremname}{testing}
\newenvironment{named}[1]{\renewcommand{\theoremname}{#1}\begin{namedtheorem}}{\end{namedtheorem}}

\theoremstyle{definition}
\newtheorem{define}[theorem]{Definition}
\newtheorem{question}[theorem]{Question}
\newtheorem{remark}[theorem]{Remark}

\title[Geometrically and diagrammatically maximal knots]{Geometrically and diagrammatically maximal knots}
\author[A.\ Champanerkar]{Abhijit Champanerkar}
\address{Department of Mathematics, College of Staten Island \& The Graduate Center, City University of New York, New York, NY, USA}
\email{abhijit@math.csi.cuny.edu}
\author[I. \ Kofman]{Ilya Kofman}
\address{Department of Mathematics, College of Staten Island \& The Graduate Center, City University of New York, New York, NY, USA}
\email{ikofman@math.csi.cuny.edu}
\author[J. \ Purcell]{Jessica S.\ Purcell}
\address{School of Mathematical Sciences, 9 Rainforest Walk, Monash University, Victoria 3800, Australia}
\email{jessica.purcell@monash.edu}

\subjclass[2010]{57M25, 57M50}

\begin{document}

\begin{abstract}
The ratio of volume to crossing number of a hyperbolic knot is known
to be bounded above by the volume of a regular ideal octahedron, and a
similar bound is conjectured for the knot determinant per crossing.
We investigate a natural question motivated by these bounds: For which
knots are these ratios nearly maximal?  We show that many families of
alternating knots and links simultaneously maximize both ratios.
\end{abstract}

\maketitle

\section{Introduction}
Despite many new developments in the fields of hyperbolic geometry,
quantum topology, 3--manifolds, and knot theory, there remain notable
gaps in our understanding about how the invariants of knots and links that
come from these different areas of mathematics are related to each
other.  In particular, significant recent work has focused on
understanding how the hyperbolic volume of knots and links is related
to diagrammatic knot invariants (see, e.g.,\ \cite{InteractionsBook,
  fkp:gutsjp}).  In this paper, we investigate such relationships
between the volume, determinant, and crossing number for sequences of
hyperbolic knots and links.

For any diagram of a hyperbolic link $K$, an upper bound for the hyperbolic volume $\vol(K)$ was given by D.~Thurston by decomposing $S^3-K$ into octahedra, placing one octahedron at each crossing, and pulling remaining vertices to $\pm\infty$.
Any hyperbolic octahedron has volume bounded above by the volume of the regular ideal octahedron, $\voct \approx 3.66386$.  So if $c(K)$ is the crossing number of $K$, then
\begin{equation}\label{eq:dylan}
 \frac{\vol(K)}{c(K)} \leq \voct. 
\end{equation} 

This result motivates several natural questions about the quantity $\vol(K)/c(K)$, which we call the {\em volume density} of $K$. How sharp is the bound of equation~\eqref{eq:dylan}?  For which links is the volume density very near $\voct$? In this paper, we address these questions from several different directions, and present several conjectures motivated by our work.

We also investigate another notion of density for a knot or link.
For any non-split link $K$, we say that $2\pi\log\det(K)/c(K)$ is its {\em determinant density}.
The following conjectured upper bound for the determinant density is equivalent to a conjecture of Kenyon for planar graphs (Conjecture \ref{conj:kenyon} below).
\begin{conjecture}\label{conj:diagmax}
If $K$ is any knot or link, \[ \frac{2\pi\log\det(K)}{c(K)} \leq \voct.\]
\end{conjecture}

We study volume and determinant density by considering sequences of knots and links.

\begin{define}\label{def:geommax}\label{def:diagmax}
A sequence of links $K_n$ with $c(K_n)\to \infty$ is \emph{geometrically maximal} if
\[ \lim_{n\to\infty}\frac{\vol(K_n)}{c(K_n)}=\voct. \]
Similarly, a sequence of knots or links $K_n$ with $c(K_n)\to \infty$ is {\em diagrammatically maximal} if
\[ \lim_{n\to\infty}\frac{2\pi\log\det(K_n)}{c(K_n)}=\voct. \]
\end{define}

In this paper, we find many families of geometrically and diagrammatically maximal knots and links.
Our examples are related to the \emph{infinite weave} $\W$, which we define
to be the infinite alternating link with the square grid
projection, as in Figure~\ref{fig:infweave}. We will see in
Section~\ref{sec:weave} that there is a complete hyperbolic structure
on $\RR^3-\W$ obtained by tessellating the manifold by regular ideal
octahedra such that the volume density of $\W$ is exactly $\voct$.
Therefore, a natural place to look for geometrically maximal
knots is among those with geometry approaching $\RR^3-\W$. We will see
that links whose diagrams converge to the diagram of $\W$ in an
appropriate sense are both geometrically and diagrammatically
maximal. To state our results, we need to define the convergence of
diagrams.

\begin{figure}
 \includegraphics[scale=0.35]{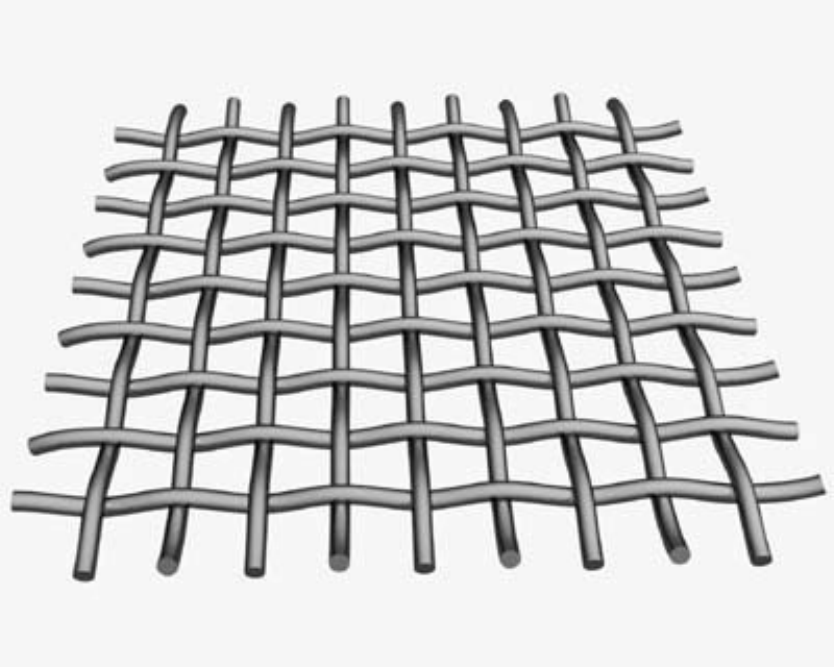}
\caption{The infinite alternating weave}
\label{fig:infweave}
\end{figure}

\begin{define}\label{def:folner}
Let $G$ be any possibly infinite graph. For any finite subgraph $H$,
the set $\partial H$ is the set of vertices of $H$ that share an edge
with a vertex not in $H$.  We let $|\cdot|$ denote the number of
vertices in a graph.  An exhaustive nested sequence of
connected subgraphs,
$\{H_n\subset G \mid \; H_n\subset H_{n+1},\; \bigcup_n H_n=G\}$, is a {\em F{\o}lner sequence} for $G$ if
\[ \lim_{n\to\infty}\frac{|\partial H_n|}{|H_n|}=0. \]
The graph $G$ is {\em amenable} if a F{\o}lner sequence for $G$ exists. In particular, the infinite square grid $G(\W)$ is amenable.  
\end{define}

For any link diagram $K$, let $G(K)$ denote the projection graph of
the diagram. We will need a particular diagrammatic condition called a
{\em cycle of tangles}, which is defined carefully in Definition
\ref{def:CycleOfTangles} below.  For an example, see Figure
\ref{fig:celtic}(a).  
We now show two strikingly similar ways to obtain geometrically and
diagrammatically maximal links.

\pagebreak

\begin{theorem}\label{thm:geommax}
Let $K_n$ be any sequence of hyperbolic alternating link diagrams that contain no cycle of tangles, such that
\begin{enumerate}
\item there are subgraphs $G_n\subset G(K_n)$ that form a F\o lner sequence for $G(\W)$, and
\item $\lim\limits_{n\to\infty} |G_n|/ c(K_n) = 1$.
\end{enumerate}
Then $K_n$ is geometrically maximal:
$\displaystyle \lim_{n\to\infty}\frac{\vol(K_n)}{c(K_n)}=\voct. $
\end{theorem}

\begin{theorem}\label{thm:diagmax}
Let $K_n$ be any sequence of alternating link diagrams such that
\begin{enumerate}
\item there are subgraphs $G_n\subset G(K_n)$ that form a F\o lner sequence for $G(\W)$, and
\item $\lim\limits_{n\to\infty} |G_n|/ c(K_n) = 1$.
\end{enumerate}
Then $K_n$ is diagrammatically maximal:
$\displaystyle \lim_{n\to\infty}\frac{2\pi\log\det(K_n)}{c(K_n)} = \voct. $
\end{theorem}

Ideas for the proof of Theorem~\ref{thm:geommax} are due to Agol.  His
unpublished results were mentioned in \cite{GarLe}, but the geometric
argument suggested in \cite{GarLe} for the lower bound is flawed
because it is based on existing volume bounds in \cite{lackenby,
  AgolStormThurston}, and these bounds only imply that for $K_n$ as
above, the asymptotic volume density lies in $[\voct/2,\,\voct]$ (see
equation~\eqref{bestbound} below). The proof of geometric maximality,
particularly the asymptotically correct lower volume bounds, uses a
combination of two main ideas.  First, we use a ``double guts'' method
to show that the volume of a link with certain diagrammatic properties
is bounded below by the volume of a right-angled polyhedron that is
combinatorially equivalent to the Menasco polyhedron of the link
(Theorem~\ref{thm:volbound}).  Second, we use the rigidity of circle
patterns associated to right-angled polyhedra to pass from ideal
polyhedra with finitely many faces to the complement of $\W$, so that
the F{\o}lner-type conditions above imply that the volume density of
$K_n$ converges to $\voct$.  Then Theorem~\ref{thm:geommax} applies to
more general links than those mentioned in \cite{GarLe}, including
links obtained by introducing crossings before taking the closure of a
finite piece of $\W$, for example as in Figure~\ref{fig:celtic}(b).

\begin{figure}
\begin{tabular}{ccc}
  \includegraphics[height=1.2in]{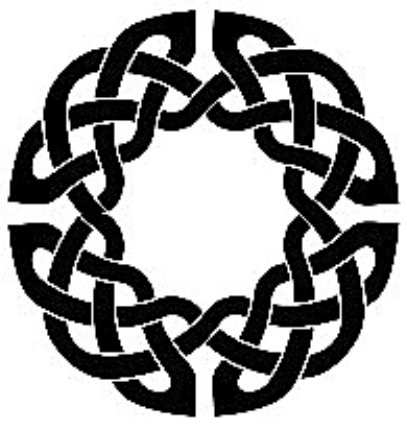} & \hspace*{0.5in} &
 \includegraphics[height=1.2in]{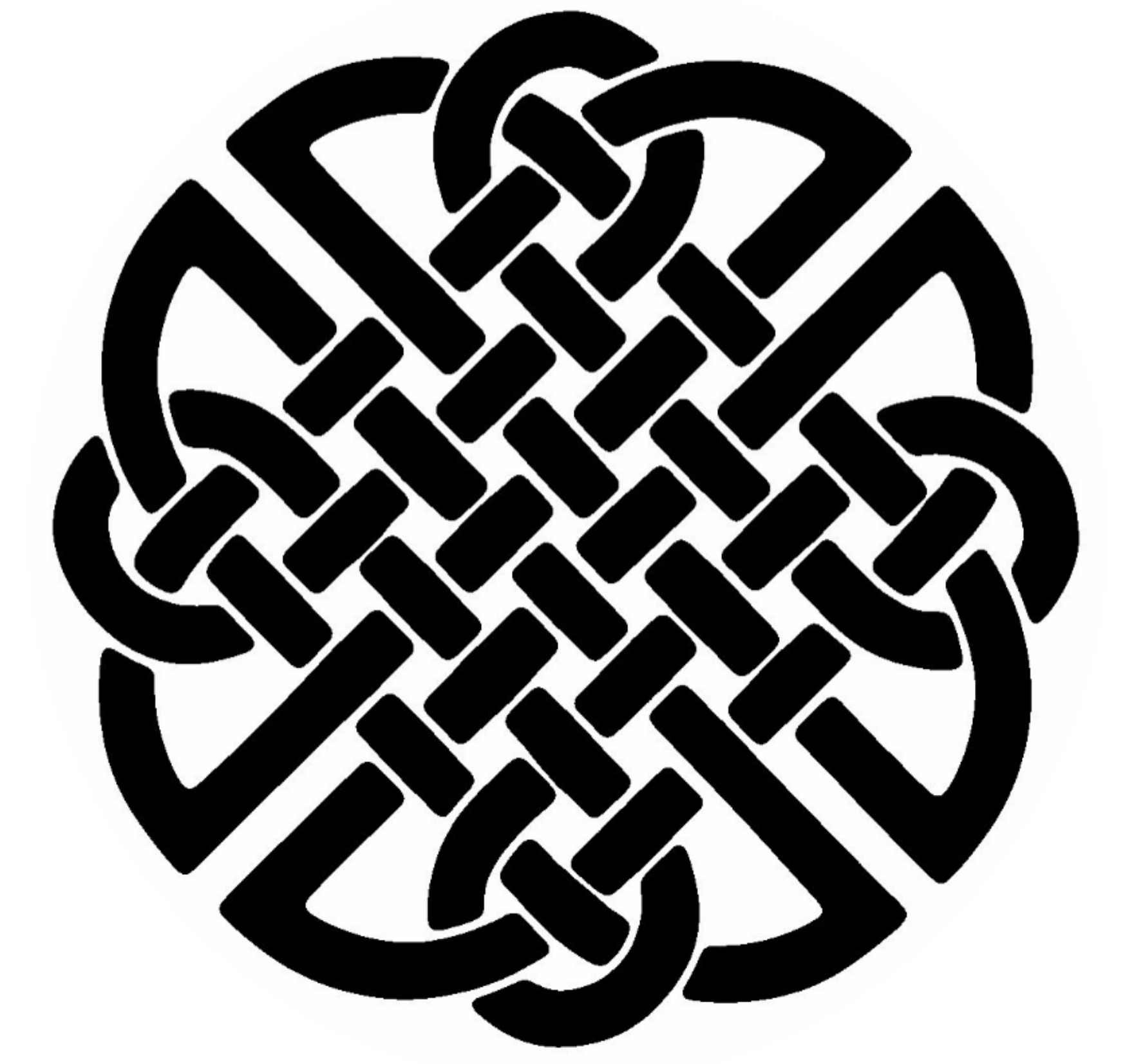} \\
(a) & \qquad & (b) \\
\end{tabular}
\caption{(a) A Celtic knot diagram that has a cycle of tangles.
  (b) A Celtic knot diagram with no cycle of tangles, which could be in a sequence that
satisfies conditions of Theorem~\ref{thm:geommax}.}
  \label{fig:celtic}
\end{figure}

Notice that any sequence of links satisfying the hypotheses of Theorem~\ref{thm:geommax} also satisfies the hypotheses of Theorem~\ref{thm:diagmax}.  This motivates the following questions.

\begin{question}\label{Q:max}
Is any diagrammatically maximal sequence of knots geometrically
maximal, and vice versa?
\[ \mbox{i.e.}\quad \lim_{n\to\infty}\frac{\vol(K_n)}{c(K_n)} = \voct  \quad \Longleftrightarrow \quad \lim_{n\to\infty}\frac{2\pi\log\det(K_n)}{c(K_n)} = \voct \ ?\]
\end{question}

Both our diagrammatic and geometric arguments below rely on special properties of alternating links.  With present tools, we cannot say much about links that are {\em mostly} alternating.

\begin{question}\label{Q:altmax}
Let $K_n$ be any sequence of links such that 
\begin{enumerate}
\item there are subgraphs $G_n\subset G(K_n)$ that form a F\o lner sequence for $G(\W)$, 
\item $K_n$ restricted to $G_n$ is alternating, and
\item $\lim\limits_{n\to\infty} |G_n|/ c(K_n) = 1$.
\end{enumerate}
Is $K_n$ geometrically and diagrammatically maximal?
\end{question}

The following family of knots and links provides an explicit example
satisfying the conditions of Theorems~\ref{thm:geommax} and
\ref{thm:diagmax}.  A {\em weaving knot} $W(p,q)$ is the alternating
knot or link with the same projection as the standard $p$--braid
$(\sigma_1\ldots\sigma_{p-1})^q$ projection of the torus knot or link
$T(p,q)$. Thus, $ c(W(p,q)) = q(p-1).$ For example, $W(5,4)$ is the
closure of the $5$--braid in Figure \ref{fig:W54}.

\begin{figure}
  \includegraphics[scale=0.15]{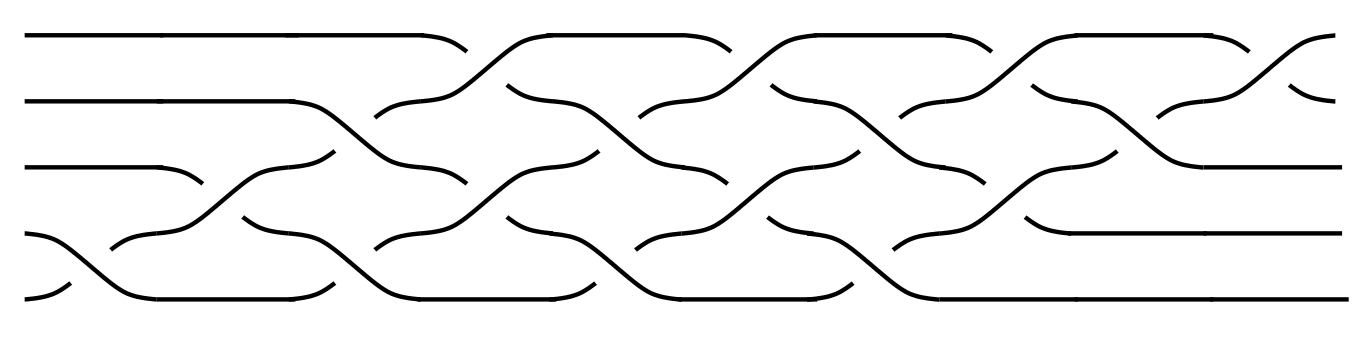}
  \caption{$W(5,4)$ is the closure of this braid.}
  \label{fig:W54}
\end{figure}

Theorems~\ref{thm:geommax} and \ref{thm:diagmax} imply that any
sequence of knots $W(p,q)$, with $p,q\to\infty$, is both geometrically
and diagrammatically maximal.  In \cite{ckp:weaving}, we provide
asymptotically sharp, explicit bounds on volumes in terms of $p$ and
$q$ alone.  Moreover, applying these asymptotically sharp bounds, we
prove in \cite{ckp:weaving} that as $p,q \to \infty$, $S^3-W(p,q)$
approaches $\R^3-\W$ as a geometric limit.  Proving that a class of
knots or links approaches $\R^3-\W$ as a geometric limit seems to be
difficult in general.  It is unknown, for example, whether all the
links of Theorem~\ref{thm:geommax} approach $\R^3-\W$ as a geometric
limit, and the proof of that theorem does not give this information.

\subsection{\bf Spectra for volume and determinant density} 
\label{sec:spectra}

We describe a more general context for Theorems
\ref{thm:geommax} and \ref{thm:diagmax}.

\begin{define}\label{def:spec}
  Let $\mathcal{C}_{\rm vol} = \{\vol(K)/c(K)\}$ and $\mathcal{C}_{\rm
    det}= \{2\pi\log\det(K)/c(K)\}$ be the sets of respective
  densities for all hyperbolic links $K$.  We define $\specvol=
  \mathcal{C}'_{\rm vol}$ and $\specdet=\mathcal{C}'_{\rm det}$ as
  their derived sets (set of all limit points).
\end{define}

The upper bound in \eqref{eq:dylan} was subsequently improved in \cite{AdamsBound} for any hyperbolic link $K$ with $c(K)\geq 5$.
Combining the lower bound in \cite{lackenby, AgolStormThurston} and the upper bound in \cite{AdamsBound}, we get the best current volume bounds for a knot or link $K$ with a prime alternating twist--reduced diagram with no bigons and $c(K)\geq 5$ crossings:
\begin{equation}\label{bestbound}
\frac{\voct}{2}\, (c(K)-2) \: \leq \: \vol(K) \: \leq \: \voct\,(c(K)-5) + 4\vtet.
\end{equation}
Here $\vtet \approx 1.01494$ is the volume of a regular ideal tetrahedron.

The upper bound in \eqref{bestbound} shows that the volume density of
any link is strictly less than $\voct$. Together with Conjecture
\ref{conj:diagmax}, this implies:
\[ \specvol, \ \specdet \subset [0,\voct] \]
For infinite sequences of alternating links without bigons,
equation~\eqref{bestbound} implies that $\specvol$ restricted 
to such links lies in $[\voct/2,\,\voct]$.

Twisting on two strands of an alternating link gives $0$ as a limit
point of both densities. Thus, we obtain the following corollary of Theorems
\ref{thm:geommax} and \ref{thm:diagmax}:
\begin{corollary}\label{cor:v8}
$\{0, \voct\} \subset \specvol \cap \specdet$.
\end{corollary}

Although $\voct$ does not occur as a volume density of any finite link,
$\voct$ is the volume density of $\W$ (see Remark \ref{rmk:v8}).
Corollary~3.7 of~\cite{ckp:weaving} also shows that $2\vtet \in \specvol$. 
It is an interesting problem to understand the sets $\specvol$,
$\specdet$, and $\specvol \cap \specdet$, and to explicitly describe
and relate their elements.

In fact, since a preprint of this paper was posted, Burton \cite{burton} and Adams et.\ al.\ \cite{Adams_density} proved that $\specvol=[0,\voct]$ and $[0,\voct] \subset \specdet$, hence $\specvol\cap\specdet = [0,\voct]$. Adams et.\ al.\ \cite{Adams_density} also showed that for any $x\in [0,\voct]$, there exists a sequence of knots $K_n$ with $x$ as a common limit point of both the volume and determinant densities of $K_n$.

\subsection{Knot determinant and hyperbolic volume}\label{sec:detvol}

There is strong experimental evidence in support of a conjectured relationship between the hyperbolic volume and the determinant of a knot, which was first observed in Dunfield's prescient online post \cite{Dunfield_website}. 
A quick experimental snapshot can be obtained from SnapPy~\cite{snappy} or Knotscape~\cite{knotscape}, which provide this data for all knots with at most $16$ crossings.  The top nine knots in this census sorted by maximum volume and by maximum determinant agree, but only set-wise!
More data and a broader context is provided by Friedl and Jackson \cite{FriedlJackson}, and Stoimenow \cite{stoimenow}.
In particular, Stoimenow \cite{stoimenow} proved there are constants $C_1,\, C_2 > 0$, such that for any hyperbolic alternating link $K$, 
\[ 2\cdot 1.0355^{\vol(K)} \: \leq \: \det(K) \: \leq \:  \left(\frac{C_1\,c(K)}{\vol(K)}\right)^{C_2\,\vol(K)} \]

Experimentally, we discovered the following surprisingly simple relationship between the two quantities that arise in the volume and determinant densities.  We have verified the following conjecture for all alternating knots up to 16 crossings, and weaving knots and links for $3\leq p\leq 50$ and $2\leq q\leq 50$.
\begin{conjecture}[Vol-Det Conjecture]\label{conj:detvol}
For any alternating hyperbolic link $K$, 
\[ \vol(K) < 2\pi\log\det(K). \]
\end{conjecture}

Conjectures \ref{conj:detvol} and \ref{conj:diagmax} would imply one direction
of Question \ref{Q:max}, that any geometrically maximal sequence of knots is diagrammatically maximal.
In contrast, we can obtain $K_n$ by twisting on two strands, such that $\vol(K_n)$ is bounded but $\det(K_n)\to\infty$.

Our main results imply that the constant $2\pi$ in Conjecture~\ref{conj:detvol} is sharp:
\begin{corollary}\label{cor:detvol}
If $\alpha<2\pi$ then there exist alternating hyperbolic knots $K$ such that\\ $\alpha\log\det(K)<\vol(K)$.
\end{corollary}
\begin{proof}
Let $K_n$ be a sequence of knots that is both geometrically and diagrammatically maximal.
Then $\displaystyle \lim_{n\to\infty}\frac{\alpha\log\det(K_n)}{c(K_n)} = \alpha\, \voct/2\pi < \voct$ and $\displaystyle \lim_{n\to\infty}\frac{\vol(K_n)}{c(K_n)}=\voct.$
Hence, for $n$ sufficiently large, $\alpha\log\det(K_n)<\vol(K_n)$.
\end{proof}

Our focus on geometrically and diagrammatically maximal knots and links naturally emphasizes the importance of alternating links. Every non-alternating link can be viewed as a modification of a diagram of an alternating link with the same projection, by changing crossings. This modification affects the determinant as follows.


\begin{prop}\label{prop:det_alt}
Let $K$ be a reduced alternating link diagram, and let $K'$ be obtained by changing any proper subset of crossings of $K$.  Then
\[ \det(K') < \det(K). \]
\end{prop}

What happens to volume under this modification? 
Motivated by Proposition~\ref{prop:det_alt}, the first two authors previously conjectured that alternating diagrams also maximize hyperbolic volume in a given projection. They have verified part (a) of the following conjecture for all alternating knots up to 18 crossings ($\approx 10.7$ million knots).

\begin{conjecture}\label{conj:alt}
(a) Let $K$ be an alternating hyperbolic knot, and let $K'$ be obtained by changing any crossing of $K$.  Then
\[ \vol(K') < \vol(K). \]
(b) The same result holds if $K'$ is obtained by changing any proper subset of crossings of $K$.
\end{conjecture}

Note that by Thurston's Dehn surgery theorem, the volume converges from below
when twisting two strands of a knot, so $\vol(K) - \vol(K')$ can be an
arbitrarily small positive number.

A natural extension of Conjecture~\ref{conj:detvol} to any hyperbolic
knot is to replace the determinant with the rank of the reduced
Khovanov homology $\widetilde{H}^{*,*}(K)$.  
Let $K$ be an alternating hyperbolic knot, and let $K'$ be obtained by changing any proper subset of crossings of $K$.
It follows from results in \cite{KH} that 
\[ \det(K') \leq \text{rank}(\widetilde{H}^{*,*}(K'))\leq \det(K). \]
Conjectures~\ref{conj:detvol} and~\ref{conj:alt} would imply that
$\vol(K') < 2\pi\log\det(K)$, but using data from KhoHo \cite{KhoHo}
we have verified the following stronger conjecture for all
non-alternating knots with up to 15 crossings. 

\begin{conjecture}\label{conj:KHvol}
For any hyperbolic knot $K$, 
\[ \vol(K) < 2\pi\log\text{\rm rank}(\widetilde{H}^{*,*}(K)). \]
\end{conjecture}
\noindent
Note that Conjecture~\ref{conj:detvol} is a special case of Conjecture~\ref{conj:KHvol}.

\subsection{Organization}
In Section~\ref{sec:diag}, we give the proof of Theorem~\ref{thm:diagmax}. 
We discuss the geometry of the infinite weave $\W$ in Section~\ref{sec:weave}.  
In Section~\ref{sec:guts}, we begin the proof of Theorem~\ref{thm:geommax} by proving that these links have volumes bounded below by the volumes of certain right-angled hyperbolic polyhedra.
In Section~\ref{sec:vollims}, we complete the proof of Theorem~\ref{thm:geommax} essentially using the rigidity of circle patterns associated to the right-angled polyhedra.
We will assume throughout that our links are non-split.

\subsection{Acknowledgments} We thank Ian Agol for sharing his ideas for the proof that closures of subsets of $\W$ are geometrically maximal, and for other helpful conversations.  We also thank Craig Hodgson for helpful conversations.  We thank Marc Culler and Nathan Dunfield for help with SnapPy~\cite{snappy}, which has been an essential tool for this project.  The first two authors acknowledge support by the Simons Foundation and PSC-CUNY.  The third author acknowledges support by the National Science Foundation under grant number DMS--1252687, and the Australian Research Council under grant DP160103085. 

\section{Diagrammatically maximal knots and spanning trees}\label{sec:diag}
In this section, we first give the proof of Theorem~\ref{thm:diagmax}, then discuss conjectures related to Conjecture~\ref{conj:diagmax}.

For any connected link diagram $K$, we can associate a connected graph $G_K$, called the Tait graph of $K$, by checkerboard coloring complementary regions of $K$, assigning a vertex to every shaded region, an edge to every crossing and a $\pm$ sign to every edge as follows:
\[
\includegraphics[height=1cm]{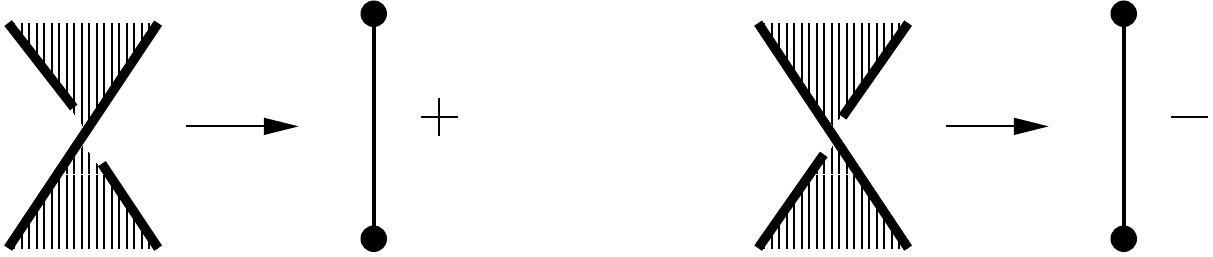}
\]
Thus, $e(G_K)=c(K)$, and the signs on the edges are all equal if and only if $K$ is alternating. So any alternating knot or link $K$ is determined up to mirror image by its unsigned Tait graph $G_K$.  

Let $\tau(G_K)$ denote the number of spanning trees of $G_K$. For any alternating link, $\tau(G_K)=\det(K)$, the determinant of $K$.  More generally, for links including non-alternating links, we have the following.

\begin{lemma}[\cite{qatwist}]\label{det_st}
For any spanning tree $T$ of $G_K$, let $\sigma(T)$ be the number of positive edges in $T$.  Let $s_\sigma(K)=\#\{$spanning trees $T$ of $G_K\ |\ \sigma(T)=\sigma\}$.  Then
\[ \det(K) = \left|\sum_\sigma (-1)^\sigma\, s_\sigma(K) \right|. \]
\end{lemma}

With this notation, we can prove Proposition~\ref{prop:det_alt} from the introduction.

\begin{named}{Proposition~\ref{prop:det_alt}}
Let $K$ be a reduced alternating link diagram, and $K'$ be obtained by changing any proper subset of crossings of $K$.  Then
\[ \det(K') < \det(K). \]
\end{named}

\begin{proof}
First, suppose only one crossing of $K$ is switched, and let $e$ be the corresponding edge of $G_K$, which is the only negative edge in $G_{K'}$.
Since $K$ has no nugatory crossings, $e$ is neither a bridge nor a loop.  
Hence, there exist spanning trees $T_1$ and $T_2$ such that $e\in T_1$ and $e\notin T_2$.
The result now follows by Lemma \ref{det_st}.

When a proper subset of crossings of $K$ is switched, by Lemma~\ref{det_st} it suffices to show that  if $s_\sigma(K)\neq 0$ then $s_{\sigma+1}(K)\neq 0$ or $s_{\sigma-1}(K)\neq 0$.  Since there are no bridges or loops, every pair of edges is contained in a cycle.  So for any spanning tree $T_1$ with $\sigma(T_1)<e(T_1)$, we can find a pair of edges $e_1$ and $e_2$ with opposite signs, such that $e_1\in {\rm{cyc}}(T_1,\,e_2)$, where recall ${\rm{cyc}}(T_1,\,e_2)$ is the set of edges in the unique cycle of $T_1\cup e_2$. It follows that $T_2 = (T_1-e_1)\cup e_2$ satisfies $\sigma(T_2)=\sigma(T_1)\pm 1$.
\end{proof}

We now show how Theorem~\ref{thm:diagmax} follows from previously known results about the asymptotic enumeration of spanning trees of finite planar graphs.  

\begin{theorem}\label{thm:diagmax2}
Let $H_n$ be any F{\o}lner sequence for the square grid, and let $K_n$ be any sequence of alternating links with corresponding Tait graphs $G_n\subset H_n$, such that
\[ \lim_{n\to\infty}\frac{\# \{x \in V(G_n):\; \deg(x)=4\} }{|H_n|} = 1, \]
where $V(G_n)$ is the set of vertices of $G_n$. Then
\[ \lim_{n\to\infty}\frac{2\pi\log\det(K_n)}{c(K_n)} = \voct. \]
\end{theorem}

\begin{proof}
Burton and Pemantle (1993), Shrock and Wu (2000), and others (see \cite{Lyons} and references therein) computed the spanning tree entropy of graphs $H_n$ that approach $G(\W)$:
\[
\lim_{n\to\infty}\frac{\log\tau(H_n)}{|H_n|} = 4{\rm C}/\pi.
\]
where ${\rm{C}} \approx 0.9160$ is Catalan's constant.  The spanning tree entropy of $G_n$ is the same as for graphs $H_n$ that approach $G(\W)$ by \cite[Corollary~3.8]{Lyons}.  Since $4{\rm{C}} =\voct$ and by the two--to--one correspondence for edges to vertices of the square grid, the result follows.
\end{proof}

Note that the subgraphs $G_n$ in Theorem \ref{thm:diagmax2} have small boundary (made precise in \cite{Lyons}) but they need not be nested, and need not exhaust the infinite square grid $G(\W)$. Because the Tait graph $G_{\W}$ is isomorphic to $G(\W)$, these results about the spanning tree entropy of Tait graphs $G_K$ are the same as for projection graphs $G(K)$ used in Theorem~\ref{thm:diagmax}.  Thus, Theorem~\ref{thm:diagmax2} implies Theorem~\ref{thm:diagmax}. This concludes the proof of Theorem~\ref{thm:diagmax}. 

\subsection{Determinant density}
We now return our attention to Conjecture~\ref{conj:diagmax} from the introduction. That conjecture is equivalent to the following conjecture due to Kenyon.

\begin{conjecture}[Kenyon \cite{kenyon}]\label{conj:kenyon}
If $G$ is any finite planar graph, 
\[ \frac{\log\tau(G)}{e(G)}\leq 2{\rm C}/\pi\approx 0.58312 \]
where ${\rm{C}} \approx 0.9160$ is Catalan's constant.
\end{conjecture}

The equivalence can be seen as follows. 
Since $4{\rm{C}} =\voct$ and $\tau(G_K)=\det(K)$, Conjecture~\ref{conj:kenyon} would immediately imply that Conjecture~\ref{conj:diagmax} holds for all alternating links $K$.  If $K$ is not alternating, then there exists an alternating link with the same crossing number and strictly greater determinant by Proposition~\ref{prop:det_alt}.  Therefore, Conjecture~\ref{conj:kenyon} would still imply Conjecture~\ref{conj:diagmax} in the non-alternating case.

On the other hand, any finite planar graph is realized as the Tait graph of an alternating link, with edges corresponding to crossings.  Hence Conjecture~\ref{conj:diagmax} implies Conjecture~\ref{conj:kenyon}.

Currently, the best proven upper bound for the determinant density is
due to Stoimenow \cite{stoimenow:maxdet}.  Let $\delta\approx 1.8393$
be the real positive root of $x^3-x^2-x-1=0$. Then
\cite[Theorem~2.1]{stoimenow:maxdet} implies that
$\frac{2\pi\log\det(K)}{c(K)}\leq 2\pi\log(\delta)\approx 3.8288$.  We
thank Jun Ge for informing us of this result. Note that planarity is
required to prove Conjecture~\ref{conj:diagmax} because Kenyon has
informed us that $\frac{2\pi\log\tau(G)}{e(G)}>\voct$ does occur for
some non-planar graphs.

\section{Geometry of the infinite weave}\label{sec:weave}

In this section, we discuss the geometry and topology of the infinite weave $\W$ and its complement $\RR^3-\W$. 
Recall that $\W$ is the infinite alternating link whose diagram projects to the square grid, as in Figure~\ref{fig:infweave}.

\begin{theorem}\label{thm:weave}
$\R^3-\W$ has a complete hyperbolic structure with a fundamental domain tessellated by regular ideal octahedra, one for each square of the infinite square grid. 
\end{theorem}

\begin{proof}
First, we view $\RR^3$ as $\RR^2 \times (-1, 1)$, with the plane of projection for $\W$ the plane $\RR^2\times\{0\}$. Thus, $\W$ lies in a small neighborhood of $\RR^2\times\{0\}$ in $\RR^2\times (-1, 1)$. We can arrange the diagram so that $\RR^3-\W$ is biperiodic and equivariant under a $\ZZ\times \ZZ$ action given by translations along the $x$ and $y$--axes, translating by two squares in each direction to match the alternating property of the diagram. Notice that the quotient gives an alternating link in the thickened torus, with fundamental region as in Figure~\ref{fig:link-on-torus}(a). A thickened torus, in turn, is homeomorphic to the complement of the Hopf link in $S^3$. Thus the quotient of $\RR^3-\W$ under the $\ZZ\times \ZZ$ action is the complement of a link $L$ in $S^3$, as in Figure~\ref{fig:link-on-torus}(b).

\begin{figure}
\begin{tabular}{ccc}
\includegraphics{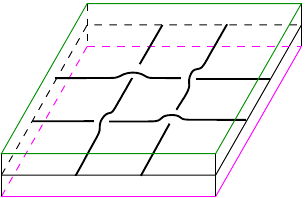} &
\includegraphics{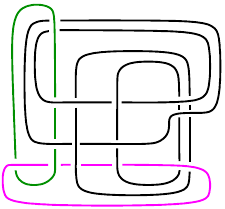} &
\includegraphics{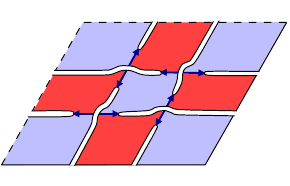} \\
& \\
(a) & (b) & (c)
\end{tabular}
  \caption{(a) The $\Z \times \Z$ quotient of $\R^3-\W$ is a link complement in a thickened torus (b) This is also the complement of link $L$ in $S^3$ (c) Cutting along checkerboard surfaces, viewed from above the projection plane.
}
  \label{fig:link-on-torus}
\end{figure}

The link complement $S^3-L$ can be easily shown to be obtained by gluing four regular ideal octahedra, for example by computer using Snap~\cite{snap} 
(which uses exact arithmetic).
Below, we present an explicit geometric way to obtain this decomposition.

Consider the two surfaces of $S^3-L$ on the projection plane of the
thickened torus, i.e.\ the image of $\RR^2\times\{0\}$. These can be
checkerboard colored on $T^2 \times \{ 0\}$.  These intersect in four
crossing arcs, running between crossings of the single square shown in
the fundamental domain of Figure~\ref{fig:link-on-torus}. Generalizing
the usual polyhedral decomposition of alternating links, due to
Menasco~\cite{menasco:polyhedra} (see also \cite{lackenby}), cut along
these checkerboard surfaces. When we cut, the manifold falls into two
pieces $X_1$ and $X_2$, each homeomorphic to $T^2\times I$, with one
boundary component $T^2\times\{1\}$, say, coming from a Hopf link
component in $S^3$, and the other now given faces, ideal edges, and
ideal vertices from the checkerboard surfaces, as follows.
\begin{enumerate}
\item For each piece $X_1$ and $X_2$, there are four faces total, two red and two blue, all quadrilaterals coming from the checkerboard surfaces.
\item There are four equivalence class of edges, each corresponding to a crossing arc.
\item Ideal vertices come from remnants of the link in $T^2\times\{0\}$: either overcrossings in the piece above the projection plane, or undercrossings in the piece below.
\end{enumerate}
The faces (red and blue), edges (dark blue), and ideal vertices (white) for $X_1$ above the projection plane are shown in Figure~\ref{fig:link-on-torus}(c).

Now, for each ideal vertex on $T^2\times\{0\}$ of $X_i$, $i=1,2$, add an edge running vertically from that vertex to the boundary component $T^2\times\{1\}$. Add triangular faces where two of these new edges together bound an ideal triangle with one of the ideal edges on $T^2\times\{1\}$. These new edges and triangular faces cut each $X_i$ into four square pyramids. Since $X_1$ and $X_2$ are glued across the squares at the base of these pyramids, this gives a decomposition of $S^3-L$ into four ideal octahedra, one for each square region in $T^2\times\{0\}$. 

Give each octahedron the hyperbolic structure of a hyperbolic regular ideal octahedron. Note that each edge meets exactly four octahedra, and so the monodromy map about each edge is the identity. Moreover, each cusp is tiled by Euclidean squares, and inherits a Euclidean structure in a horospherical cross--section. Thus by the Poincar\'e polyhedron theorem (see, e.g., \ \cite{epstein-petronio}), this gives a complete hyperbolic structure on $S^3-L$. 

Thus the universal cover of $S^3-L$ is $\HH^3$, tiled by regular ideal octahedra, with a square through the center of each octahedron projecting to a square from the checkerboard decomposition. 
Taking the cover of $S^3-L$ corresponding to the $\ZZ\times\ZZ$
subgroup associated with the Hopf link, we obtain a complete
hyperbolic structure on $\RR^3-\W$, with a fundamental domain
tessellated by regular ideal octahedra, with one octahedron for each
square of the square grid, as claimed.
\end{proof}

\begin{figure}[h]
\begin{tabular}{ccc}
 \includegraphics[scale=0.5]{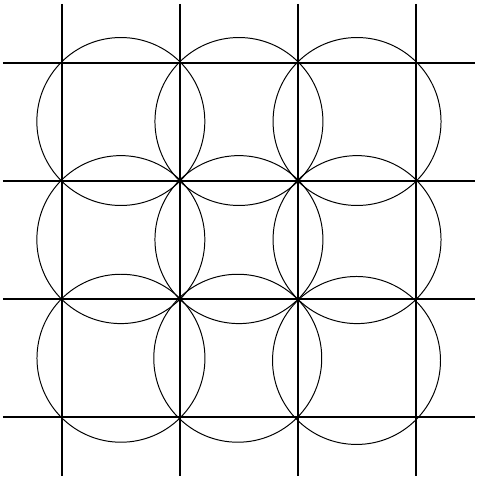} & \qquad &
 \includegraphics[scale=1.1]{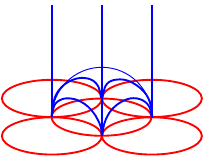} \\
(a) & \qquad & (b) \\
\end{tabular}
\caption{(a) Circle pattern for hyperbolic planes of the top piece $X_1$ of $\R^3-\W$.
  (b) Hyperbolic planes bounding one top square pyramid.}
  \label{fig:W_top}
\end{figure}

\begin{figure}[h]
\includegraphics[scale=0.22]{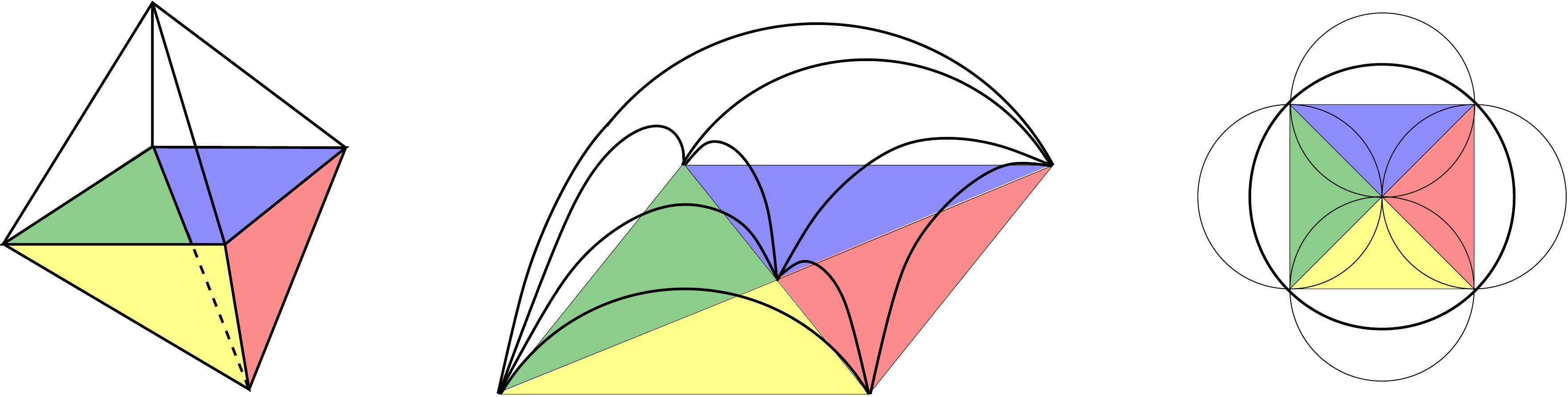} 
\caption{A regular ideal hyperbolic octahedron is obtained by gluing the
  two square pyramids. Hyperbolic planes that form the bottom square pyramid and the associated circle pattern are shown.}
  \label{fig:W_bottom}
\end{figure}

\begin{figure}
  \input{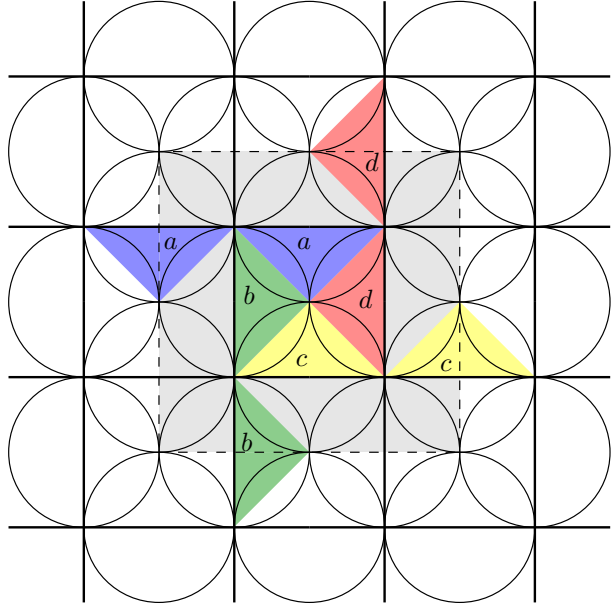} \\
  \caption{Face pairings for a fundamental domain $\P_{\W}$ of $\R^3-\W$. The 
shaded part indicates the fundamental domain of $S^3-L$. }
  \label{fig:weave-decom}
\end{figure}

The proof of Theorem~\ref{thm:weave} also provides the face
pairings for the regular ideal octahedra that tessellate the 
fundamental domain for $\RR^3-\W$.
We also discuss the associated circle patterns on $T^2\times\{0\}$,
which in the end play an important role in the proof of
Theorem~\ref{thm:geommax} in Section \ref{sec:vollims}.

A regular ideal octahedron is obtained by gluing two square pyramids,
which we will call the {\em top} and {\em bottom} pyramids. 
In Figure~\ref{fig:W_top}(b), the apex of
the top square pyramid is at infinity, the triangular faces are shown
in the vertical planes, and the square face is on the hemisphere.  

From the proof above, $\RR^3-\W$ is cut into $\tilde{X_1}$ and
$\tilde{X_2}$, such that $\tilde{X_1}$ is obtained by gluing top
pyramids along triangular faces, and $\tilde{X_2}$ by gluing bottom
pyramids along triangular faces.  The circle pattern in
Figure~\ref{fig:W_top}(a) shows how the square pyramids in $\tilde{X_1}$ 
are viewed from infinity on the $xy$-plane.

Similarly, in Figure~\ref{fig:W_bottom}, we show the hyperbolic planes
that form the bottom square pyramid, and the associated circle
pattern.  In this figure, the apex of the bottom square pyramid is in
the center, the triangular faces are on hemispherical planes, and the
square face is on the upper hemisphere.  The circle pattern shows how the bottom square pyramids on the $xy$-plane are viewed from infinity on the $xy$-plane.

A fundamental domain $\P_{\W}$ for $R^3-\W$ in $\HH^3$ is explicitly
obtained by attaching each top pyramid of $\tilde{X_1}$ to a bottom
pyramid of $\tilde{X_2}$ along their common square face.  Hence,
$\P_{\W}$ is tessellated by regular ideal octahedra. By the proof above,
an appropriate $\pi/2$ rotation is needed when gluing the square
faces, which determines how adjacent triangular faces are glued to
obtain $\P_{\W}$.  Figure~\ref{fig:weave-decom} shows the face pairings
for the triangular faces of the bottom square pyramids, and the
associated circle pattern. The face pairings are equivariant under the
translations $(x,y) \mapsto (x\pm 1,y\pm 1)$. That is, when a pair of
faces is identified, then the corresponding pair of faces under this
translation is also identified.

\begin{remark}\label{rmk:v8}
  Because every regular ideal octahedron corresponds to a square face
  that meets four crossings, and any crossing meets four square faces
  that correspond to four ideal octahedra, it follows that the volume
  density of the infinite link $\W$ is exactly $\voct$.
\end{remark}

\section{Guts and lower volume bounds}\label{sec:guts}

In this section, we begin the proof of Theorem~\ref{thm:geommax} by showing that knots satisfying the hypotheses of that theorem have volume bounded below by the volume of a certain right--angled polyhedron. The main result of this section is Theorem~\ref{thm:volbound} below.

The techniques we use to bound volume from below involve guts of embedded essential surfaces, which we define below. Since we will be dealing with orientable as well as nonorientable surfaces, we say that any surface is \emph{essential} if and only if the boundary of a regular neighborhood of the surface is an essential (orientable) surface, i.e.\ it is incompressible and boundary incompressible.

If $N$ is a 3--manifold admitting an embedded essential surface $\Sigma$, then $N\cut\Sigma$ denotes the manifold with boundary obtained by removing a regular open neighborhood of $\Sigma$ from $N$.  Let $\widetilde{\Sigma}$ denote the boundary of $N\cut\Sigma$, which is homeomorphic to the unit normal bundle of $\Sigma$.  Note that if $N$ is an open manifold, i.e.\ $N$ has nonempty topological frontier consisting of rank--2 cusps, then $\widetilde{\Sigma}$ will be a strict subset of the topological frontier of $N\cut\Sigma$, which consists of $\widetilde{\Sigma}$ and a collection of tori and annuli coming from cusps of $N$.

\begin{define}\label{def:ParabolicLocus}
The \emph{parabolic locus} of $N\cut\Sigma$ consists of tori and annuli on the topological frontier of $N\cut\Sigma$ which come from cusps of $N$.
\end{define}

We let $D(N\cut\Sigma)$ denote the double of the manifold
$N\cut\Sigma$, doubled along the boundary $\widetilde{\Sigma}$.  The
manifold $D(N\cut\Sigma)$ admits a JSJ--decomposition.  That is, it
can be decomposed along essential annuli and tori into Seifert fibered
and hyperbolic pieces.  This gives an annulus decomposition of
$N\cut\Sigma$: a collection of annuli in $N\cut\Sigma$, disjoint from
the parabolic locus, that cut $N\cut\Sigma$ into $I$--bundles, Seifert
fibered solid tori, and \emph{guts}.  Let $\guts(N\cut\Sigma)$ denote
the guts, which is the portion that admits a hyperbolic metric with
geodesic boundary. Let $D(\guts(N\cut\Sigma))$ denote the complete hyperbolic
3-manifold obtained by doubling the $\guts(N\cut\Sigma)$ along 
the part of boundary contained in $\widetilde{\Sigma}$ (i.e. disjoint from 
the parabolic locus of $N\cut \Sigma$).

\begin{theorem}[Agol--Storm--Thurston \cite{AgolStormThurston}]\label{thm:AST}
Let $N$ be a finite volume hyperbolic manifold, and $\Sigma$ an embedded $\pi_1$--injective surface in $N$.  Then
\begin{equation}\label{eqn:AST}
  \vol(N) \: \geq \: \frac{1}{2}\,\vtet\,||D(N\cut\Sigma)|| \: = \:
  \frac{1}{2}\vol(D(\guts(N\cut\Sigma))).
\end{equation}
Here the value $||\cdot||$ denotes the Gromov norm of the manifold. 
\end{theorem}

We will prove Theorem \ref{thm:geommax} in a sequence of lemmas that concern the geometry and topology of alternating links, and particularly ideal checkerboard polyhedra that make up the complements of these alternating links.  These ideal polyhedra were described by Menasco \cite{menasco:polyhedra} (see also \cite{lackenby}).  We review them briefly.

\begin{define}\label{def:polyhedra}
Let $K$ be a hyperbolic alternating link with an alternating diagram (also denoted $K$) that is checkerboard colored.  Let $B$ (blue) and $R$ (red) denote the checkerboard surfaces of $K$.  If we cut $S^3-K$ along both $B$ and $R$, the manifold decomposes into two identical ideal polyhedra, denoted by $P_1$ and $P_2$.  We call these the \emph{checkerboard ideal polyhedra of $K$}.  They have the following properties.
\begin{enumerate}
\item\label{item:poly1} For each $P_i$, the ideal vertices and edges form a 4--valent graph on $\partial P_i$, and that graph is isomorphic to the projection graph of $K$ on the projection plane.
\item\label{item:poly2} The faces of $P_i$ are colored blue and red corresponding to the checkerboard coloring of $K$.
\item\label{item:poly3} To obtain $S^3-K$ from $P_1$ and $P_2$, glue each red face of $P_1$ to the same red face of $P_2$, and glue each blue face of $P_1$ to the same blue face of $P_2$.
\end{enumerate}
The gluing maps in item~\eqref{item:poly3} are not the identity maps, but rather involve a single clockwise or counterclockwise ``twist'' (see \cite{menasco:polyhedra} for details).  In this paper, we won't need the precise gluing maps, just which faces are attached.
\end{define}

The checkerboard surfaces $B$ and $R$ are well known to be essential in the alternating link complement \cite{menasco-thist:alternating}.  We will cut along these surfaces, and investigate the manifolds $(S^3-K)\cut B$ and $(S^3-K)\cut R$.  Note that because $K\subset B$, there is a homeomorphism $(S^3-K)\cut B \cong S^3\cut B$, with parabolic locus mapping to identical parabolic locus.  We will simplify notation by writing $S^3\cut B$.

We will need to work with diagrams without Conway spheres.
Menasco~\cite{menasco} and Thistlethwaite~\cite{Thistlethwaite_PJM, Thistlethwaite_JKTR} showed that for a prime alternating link diagram, essential Conway spheres can appear in very limited ways. We also consider inessential 4--punctured spheres, for example bounding rational tangles. 
Following Thistlethwaite's notation, we say that a 4--punctured sphere is {\em visible} if it is parallel to one dividing the diagram into two tangles, as in Figure~\ref{fig:visibleTangle}.
In \cite[Proposition 5.1]{Thistlethwaite_PJM}, it was shown that if there is an essential visible Conway sphere, then it is always visible in any prime alternating diagram of the same link. 

\begin{figure}
\includegraphics{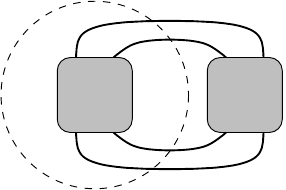}
  \caption{A visible Conway sphere}
  \label{fig:visibleTangle}
\end{figure}

\begin{define}\label{def:CycleOfTangles}
We call a 2--tangle {\em knotty} if it is nontrivial, and not a (portion of a) single twist region; i.e.\ not a rational tangle of type $n$ or $1/n$ for $n\in\Z$.  We will say that $K_n$ contains a {\em cycle of tangles} if $K_n$ contains a visible Conway sphere with a knotty tangle on each side.
\end{define}

For any link that contains a cycle of tangles, one of its two Tait graphs has a $2$--vertex cut set coming from the regions on either side of a tangle.
On the other hand, the Tait graphs of $\W$ are both the square grid, which is $4$--connected.  So $\W$ has no cycle of tangles.

Recall that a diagram is \emph{prime} if any simple closed curve meeting the diagram exactly twice, transversely in the interiors of edges, contains no crossings on one side. A diagram is \emph{twist--reduced} if any simple closed curve meeting the diagram exactly twice in crossings, running directly through the crossing to the region on the opposite side, bounds a (possibly empty) string of bigon regions of the diagram on one side.

\begin{lemma}\label{lemma:remove-bigons}
Let $K$ be a link diagram that is prime, alternating, and twist--reduced with no cycle of tangles, with red and blue checkerboard surfaces.  Obtain a new link diagram $K_R$ (resp. $K_B$) by removing red (resp. blue) bigons from the diagram of $K$ and replacing adjacent red (resp. blue) bigons in a twist region with a single crossing in the same direction.  Then the resulting diagram $K_R$ (resp. $K_B$) is prime, alternating, and twist--reduced with no cycle of tangles.
\end{lemma}

\begin{proof}
Because twist regions bounding red bigons are replaced by a single crossing in the same direction, the diagram of $K_R$ remains alternating.  If it is not prime, there would be a simple closed curve $\gamma$ meeting the diagram transversely twice in two edges, with crossings on either side.  Because the closed curve $\gamma$ does not meet crossings, we may re-insert the red bigons into a small neighborhood of the crossings of $K_R$ without meeting $\gamma$.  Then $\gamma$ gives a simple closed curve in the diagram of $K$ meeting the diagram twice with crossings on either side, contradicting the fact that $K$ is prime.

Next suppose the diagram of $K_R$ contains a cycle of tangles.  The corresponding visible Conway sphere must avoid crossings of $K_R$, so we may re-insert red bigons into a neighborhood of crossings of the diagram, avoiding boundaries of tangles, and we obtain a visible Conway sphere in $K$.  Because $K$ contains no cycle of tangles, one of the resulting tangles must be trivial, or a single twist region.  If the tangle is trivial in $K$, then it is trivial in $K_R$.  But if a tangle is a part of a twist region in $K$, then it is either a part of a twist region in $K_R$, or a single crossing, depending on whether the bigons in the twist region are blue or red.  In either case, the Conway sphere bounds a knotty tangle in $K$, which is not allowed by Definition~\ref{def:CycleOfTangles}.

Finally suppose that $K_R$ is not twist--reduced.  Then there exists a simple closed curve $\gamma'$ meeting the diagram in exactly two crossings $x$ and $y$, with $\gamma'$ running through opposite sides of the two crossings, such that neither side of $\gamma'$ bounds a string of bigons in a twist region.  Perturb $\gamma'$ slightly so that it contains $x$ on one side, and $y$ on the other.  Then $\gamma'$ defines two nontrivial tangles, one on either side of $\gamma'$, neither of which can be knotty. This contradicts the fact shown in the previous paragraph, namely that $K$ is not a cycle of tangles. 
\end{proof}

\begin{corollary}\label{cor:remove-bigons}
For $K$ as in Lemma~\ref{lemma:remove-bigons}, let $K_{BR}$ be obtained by replacing any twist region in the diagram of $K$ by a single crossing (removing both red and blue bigons).  Then the diagram of $K_{BR}$ will be prime, alternating, and twist--reduced with no cycle of tangles.
\end{corollary}

\begin{proof}
Replace $K$ in the statement of Lemma~\ref{lemma:remove-bigons} with $K_R$ and apply the lemma to the blue bigons of $K_R$.
\end{proof}

Given a twist region in the diagram of a knot or link, recall that a \emph{crossing circle} at that twist region is a simple closed curve in the diagram, bounding a disk in $S^3$ that is punctured exactly twice by the diagram, by strands of the link running through that twist region.

\begin{lemma}\label{lemma:step1-cut1}
Let $K$ be a hyperbolic link with a prime, alternating, twist--reduced diagram (also called $K$) with no cycle of tangles.  Let $B$ denote the blue checkerboard surface of $K$.  Let $K_R$ be the link with diagram obtained from that of $K$ by replacing adjacent red bigons by a single crossing, and let $B_R$ be the blue checkerboard surface for $K_R$.
Then there exists a collection of twist regions bounding blue bigons in $K_R$, and Seifert fibered solid tori $E$, with the core of each solid torus in $E$ isotopic to a crossing circle encircling one of these twist regions, such that 
\[
\guts(S^3\cut B) \: =  \: \guts(S^3\cut B_R) \: = \: (S^3\cut B_R) - E,
\]
and
\[
\vol(S^3-K) \: \geq \: \frac{1}{2}\,\vtet\,||D(S^3\cut B_R)|| \: = \: \frac{1}{2}\vol(D((S^3\cut B_R)-E)).
\]
\end{lemma}

\begin{proof}
By Theorem~\ref{thm:AST},
\[
\vol(S^3-K) \: \geq \: \frac{1}{2}\,\vtet\, ||D((S^3\cut B)|| \: = \: \frac{1}{2}\vol(D(\guts(S^3\cut B))),
\]
so the claim about volumes follows from the claim about guts.

Lackenby notes that for a prime, twist--reduced alternating diagram $K$, $\guts((S^3-K)\cut B)$ is equal to $\guts((S^3-K_R)\cut B_R)$ \cite[Section~5]{lackenby}.  By \cite[Theorem~13]{lackenby}, $\chi(\guts(S^3\cut B_R)) = \chi(S^3\cut B_R)$, where $\chi(\cdot)$ denotes Euler characteristic.  In fact, in the proof of that theorem, Lackenby shows that a bounding annulus of the characteristic submanifold is either boundary parallel, or separates off a Seifert fibered solid torus.  We review the important features of that proof to determine the form of the Seifert fibered solid tori in the collection $E$ required by this lemma.

In the case that there is a Seifert fibered solid torus, its boundary is made up of at least one annulus on $B_R$ and at least one essential annulus $A$ in $S^3\cut B_R$.  The essential annulus $A$ is either parabolically compressible or parabolically incompressible, as defined in \cite{lackenby} (see also \cite[Definition~4.5]{fkp:gutsjp}).  If it is parabolically compressible, then it decomposes into essential product disks.  Lackenby proves in \cite[Theorem~14]{lackenby} that there are no essential product disks.  But if $A$ is parabolically incompressible, then following the proof of \cite[Theorem~14]{lackenby} carefully, we see that such a Seifert fibered solid torus determines a \emph{cycle of fused units}, as shown in Figure~\ref{fig:fused-units} (left), which is a reproduction of Figure~14 of \cite{lackenby}, with three fused units shown. More generally, there must be at least two fused units; otherwise, we have a M\"obius band and not an essential annulus by the proof of \cite[Lemma~4.1]{fkp:hyperbolic}. 

\begin{figure}
  \includegraphics{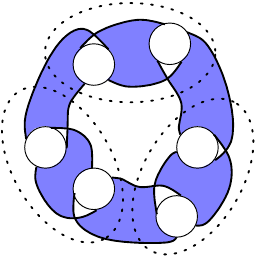} \hspace{.5in}
  \includegraphics{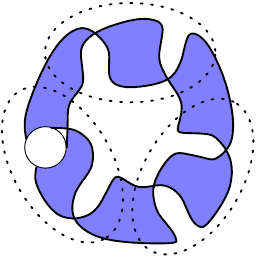}
  \caption{Left: a cycle of three fused units. Right: All but one of the tangles are trivial.}
  \label{fig:fused-units}
\end{figure}

The ellipses in dotted lines in Figure~\ref{fig:fused-units} represent the boundaries of normal squares that form the essential annulus.  Each of these encircles a \emph{fused unit}, which is made up of two crossings and a (possibly trivial) tangle, represented by a circle in the figure.  The Seifert fibered solid torus is made up of two copies of such a figure, one in each polyhedron, and consists of the region exterior to the ellipses.  That is, it meets the blue surface in strips between dotted ellipses, and meets the red surface in a disk in the center of the diagram, and one outside the diagram.  This gives a ball, with fibering of an $I$--bundle, with each interval of $I$ parallel to the blue strips and with its endpoints on the red disks.  The two balls are attached by gluing red faces, giving a Seifert fibered solid torus whose core runs through the center of the two red disks.

Now, we want to show that such a Seifert fibered solid torus only arises in a twist region of blue bigons.  Consider a single cycle of fused units.  Note that if any one of the tangles in that fused unit is non-trivial, then the boundary of the fused unit is a visible Conway sphere bounding at least one knotty tangle.  If more than one of the fused units in the cycle have this property, then by grouping other tangles in the cycle of fused units into these non-trivial tangles, we find that our diagram contains a cycle of tangles, contrary to assumption.  

So at most one of the fused units in the cycle can have a non-trivial tangle.  If both tangles in a fused unit are non-trivial, then by joining one non-trivial tangle to all other tangles in the cycle of fused units, we obtain again a cycle of tangles, contrary to assumption.

So at most one of the tangles in the cycle of fused units is non-trivial.  If all the tangles are trivial, then the diagram is that of a $(2,q)$--torus link, contradicting the fact that it is hyperbolic.  Hence exactly one of the tangles is non-trivial.  This is shown in Figure~\ref{fig:fused-units} (right).  Notice in this case, the cycle of fused units is simply a twist region of the diagram bounding blue bigons.  Notice also that the Seifert fibered solid torus has the form claimed in the statement of the lemma.

Then $\guts( S^3\cut B_R ) = (S^3\cut B_R)-E$.
\end{proof}

To simplify notation, let $M_B = D(S^3\cut B)$ and let $M_R = D(S^3\cut R)$.  

\begin{lemma}\label{lemma:step2-cut2}
Let $K$ be a link with a prime, alternating diagram with checkerboard surfaces $B$ and $R$.  Then the manifold $M_B$ contains an embedded essential surface $DR$ obtained by doubling $R$.
\end{lemma}

\begin{proof}
Note that $M_B$ has an ideal polyhedral decomposition coming from the checkerboard ideal polyhedra of Definition~\ref{def:polyhedra}.  That is, $S^3-K$ is obtained from two polyhedra $P_1$ and $P_2$ with red and blue faces glued.  The manifold $S^3\cut B$ is obtained by cutting along blue faces, or removing the gluing maps on blue faces of $P_i$.

Then the double, $M_B$, is obtained by taking two copies, $P_i^1$ and $P_i^2$, of $P_i$, gluing red faces of $P_1^i$ to red faces of $P_2^i$ by a twist, and gluing blue faces of $P_j^1$ to those of $P_j^2$ by the identity.  Note that $DR$ consists of all red faces of the four polyhedra.

Now, suppose that $DR$ is not essential.  Suppose first that the boundary of a regular neighborhood of $DR$, call it $\widetilde{DR}$, is compressible.  Let $F$ be a compressing disk for $\widetilde{DR}$.  Then $\partial F$ lies on $\widetilde{DR}$, but the interior of $F$ is disjoint from a regular neighborhood of $DR$.  Make $F$ transverse to the faces of the $P_i^j$.  Note now that $F$ must intersect $B$, else $F$ lies completely in one of the $P_i^j$, hence $F$ can be mapped into $S^3-K$ to give a compression disk for $R$ in $S^3-K$.  This is impossible since $R$ is incompressible in $S^3-K$.

Now consider the intersections of $B$ with $F$.  We may assume there are no simple closed curves of intersection, or an innermost such curve would bound a compressing disk for $B$, which we can isotope off using the fact that $B$ is essential.  Thus $B \cap F$ consists of arcs running from $\partial F$ to $\partial F$.

An outermost arc of $B\cap F$ cuts off a subdisk $F'$ of $F$ whose boundary consists of an arc on $\partial F \subset \widetilde{DR}$ and an arc on $B$.  The boundary $\partial F'$ gives a closed curve on the checkerboard colored polyhedron which meets exactly two edges and two faces.  Using the correspondence between the boundary of the polyhedron and the diagram of the link, item~\eqref{item:poly1} of Definition~\ref{def:polyhedra}, it follows that $\partial F'$ gives a closed curve on the diagram of $K$ that intersects the diagram exactly twice.  Because the diagram of $K$ is prime, there can be crossings on only one side of $F'$.  Thus the arc of $\partial F'$ on $B$ must have its endpoints on the same ideal edge of the polyhedron, and we may isotope it off, reducing the number of intersections $|F\cap B|$.  Continuing in this manner, we reduce to the case $F\cap B = \emptyset$, which is a contradiction.

The proof that $\widetilde{DR}$ is boundary incompressible follows a similar idea.  Suppose as above that $F$ is a boundary compressing disk for $\widetilde{DR}$.  Then $\partial F$ consists of an arc $\alpha$ on a neighborhood $N(K)$ of $K$, and an arc $\beta$ on $\widetilde{DR}$.  As before, $F$ must intersect $B$ or it gives a boundary compression disk for $R$ in $S^3-K$.  As before, $F$ cannot intersect $B$ in closed curves, and primality of the diagram of $K$ again implies $F$ cannot intersect $B$ in arcs that cut off subdisks of $F$ with boundary disjoint from $N(K)$.  Hence all arcs of intersection $F\cap B$ have one endpoint on $\alpha = \partial F \cap N(K)$ and one endpoint on $\beta = \partial F\cap \widetilde{DR}$. Again there must be an outermost such arc, cutting off a disk $F' \subset F$ embedded in a single ideal polyhedron with $\partial F'$ consisting of three arcs, one on $R$, one on $B$, and one on an ideal vertex of the polyhedron (coming from $N(K)$).  But then $\partial F'$ must run through a vertex and an adjacent edge, hence it can be isotoped off, reducing the number of intersections of $F$ and $B$. Repeating a finite number of times, again $B\cap F = \emptyset$, which is a contradiction.
\end{proof}

\begin{lemma}\label{lemma:step2-cutE}
Let $K_R$ be a link with a prime, twist--reduced diagram with no red bigons and no cycle of tangles, with checkerboard surfaces $B_R$ and $R_R$, and let $E$ be the Seifert fibered solid tori from Lemma~\ref{lemma:step1-cut1}.
Denote the double of the red surface in $D((S^3\cut B_R)-E)$ by $DR_E$. 
Then $DR_E$ is essential in $D((S^3\cut B_R)-E)$.
\end{lemma}

\begin{proof}
Recall that to prove this, we need to show that the boundary of a regular neighborhood of $DR_E$, call it $\widetilde{DR_E}$, is incompressible and boundary incompressible in $D((S^3\cut B_R)-E)$.

By the previous lemma, we know $\widetilde{DR_R} \subset D(S^3\cut B_R)$ is incompressible.  We will use the fact that $D((S^3\cut B_R)-E)$ is an embedded submanifold of $D(S^3\cut B_R)$, with $DR_E \subset DR_R$, and that the cores of the Seifert fibered solid tori in $E$ are all isotopic to crossing circles for the diagram, encircling blue bigons, by Lemma~\ref{lemma:step1-cut1}.  Such a crossing circle intersects the polyhedra in the decomposition of $D(S^3\cut B_R)$ in arcs with endpoints on distinct red faces. 

Now, suppose there exists a compressing disk $\Phi$ for $\widetilde{DR_E}$.  Then $\Phi \subset D((S^3\cut B_R)-E) \subset D(S^3\cut B_R)$, which has a decomposition into ideal polyhedra coming from $S^3-K$, and so we may isotope $\Phi$, keeping it disjoint from $E$, so that it meets faces and edges of the polyhedra transversely.  The boundary of $\Phi$ lies entirely on $\widetilde{DR_E}$, which we may isotope to lie entirely on the red faces of the polyhedra.  As in the proof of the previous lemma, we consider how $\Phi$ intersects blue faces.

Suppose first that $\Phi$ does not meet any blue faces.  Then $\partial \Phi$ must lie in a single red face.  But then it bounds a disk in that red face.  If it does not bound a disk in $\widetilde{DR_E}$, then that disk must meet $E$.  Then the disk meets the core of a component of $E$, which is a portion of a crossing circle in the polyhedron.  However, since $\Phi \cap E$ is empty, the entire portion of the crossing circle in the polyhedron must lie within $\bdy \Phi$, and thus have both its endpoints in the polyhedron within the disk bounded by $\bdy \Phi$.  This is a contradiction: any crossing circle meets two distinct red faces. 

So suppose $\Phi$ meets a blue face.  Then just as above, an outermost arc of intersection on $\Phi$ defines a curve on the diagram of $S^3-K$ meeting the knot exactly twice.  It must bound no crossings on one side, by primality of $K$.  Then the curve bounds a disk on the polyhedron, and a disk on $\Phi$, so either we may use these disks to isotope away the intersection with the blue face, or the disk in the polyhedron meets $E$.  But again, this implies that a portion of crossing circle lies between $\Phi$ and this disk on the polyhedron.  Again, since the crossing circle has endpoints in distinct red faces, this is impossible without intersecting $\Phi$.  This proves that $\widetilde{DR_E}$ is incompressible.

The proof of boundary incompressibility is very similar.  If there is a boundary compressing disk $\Phi$, then $\bdy \Phi$ consists of an arc $\beta$ on the boundary of $D((S^3\cut B_R) -E)$ and an arc $\alpha$ on red faces of the polyhedra.  A similar argument to that above implies that $\alpha$ cannot lie in a single red face: since $\Phi$ does not meet crossing circles at the cores of $E$, $\bdy \Phi$ would bound a disk on $\widetilde{DR_E}$.  So $\alpha$ intersects a blue face of the polyhedron.  Then consider an outermost blue arc of intersection.  As above, it cannot cut off a disk with boundary consisting of exactly one red arc and one blue arc.  So it cuts off a disk $\Phi'$ on $\Phi$ with boundary a red arc $\alpha'$, an arc $\beta'$ on the boundary of $D((S^3\cut B_R)-E)$, and a blue arc $\gamma'$.  This bounds a disk on the boundary of the polyhedron.  The interior of the disk cannot intersect $E$, else $\Phi'$ would intersect a crossing circle.  Hence we may use this disk to isotope away this intersection.  Thus $\widetilde{DR_E}$ is boundary incompressible, and the lemma holds. 
\end{proof}

\begin{lemma}\label{lemma:step2-guts-nobigons}
Let $K$ be a hyperbolic link with prime, alternating, twist--reduced diagram with no bigons and no cycle of tangles.  Let $B$ and $R$ denote its checkerboard surfaces.  Let $M_B = D(S^3\cut B)$, and $DR$ be the double of $R$ in $M_B$ as above.  Then
\[
\guts(M_B\cut DR) \: = \: M_B\cut DR.
\]
That is, in the annulus version of the JSJ decomposition of $M_B\cut DR$, there are no $I$--bundle or Seifert fibered solid torus components.
\end{lemma}

\begin{proof}
The manifold $M_B\cut DR$ is obtained by gluing four copies of the checkerboard ideal polyhedra of Definition~\ref{def:polyhedra}, by gluing $P_i^1$ to $P_i^2$ by the identity on blue faces, $i=1, 2$, and leaving red faces unglued.  

If $M_B\cut DR$ does contain $I$--bundle or Seifert fibered solid torus components, then there must be an essential annulus $A$ in $M_B\cut DR$ disjoint from the parabolic locus.  Suppose $A$ is such an annulus. Then $A$ has boundary components on $\widetilde{DR}$ and interior disjoint from a neighborhood of $R$.  Put $A$ into normal form with respect to the polyhedra $P_i^j$ of $M_B\cut DR$.  Because $A$ is essential, it must intersect $B$ in arcs running from one component of $\bdy A$ to the other, cutting $A$ into an even number of squares $S_1, S_2, \dots, S_{2n}$ alternating between $P_i^1$ and $P_i^2$, for fixed $i$.

The square $S_2$ is glued along some arc $\alpha_1$ in $A\cap B$ to the square $S_1$, and $S_2$ is glued along another arc $\alpha_2$ in $A\cap B$ to the square $S_3$.  The squares $S_1$ and $S_3$ lie in the same polyhedron $P_i^j$.  Superimpose $\bdy S_2$ on that polyhedron. Because $S_1$ and $S_3$ are glued by the identity on $B$, the arcs of all three squares coming from one component of $\bdy A$ lie on the same red face of the polyhedron.  Similarly for the other component of $\bdy A$.  The same argument applies to any three consecutive squares, showing in general that one component of $\bdy A$ lies entirely in two identical red faces of $P_i^1$ and $P_i^2$, and these are glued by the identity on adjacent blue faces. 
The result is shown in Figure~\ref{fig:cycle-tangles} (left). 

\begin{figure}
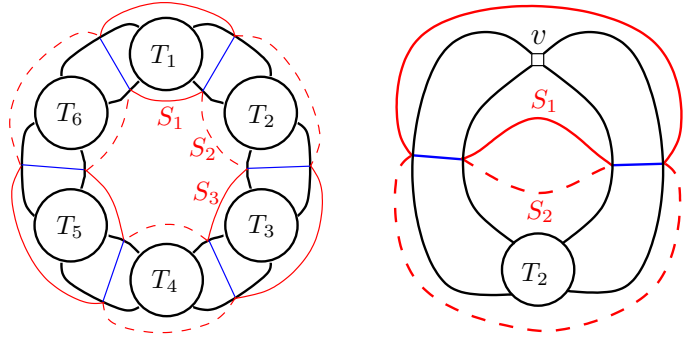

  \input{figures/cycle-tangles-sq.pspdftex}
  \hspace{.25in}
  \input{figures/cycle-tangles2Tangles.pspdftex}
  \caption{Left: Squares making up an essential annulus. Right: The case when there are only two squares. Note both can be isotoped to encircle the ideal vertex $v$ shown.}
  \label{fig:cycle-tangles}
\end{figure}

By hypothesis, we have no cycle of tangles in our diagram.  Thus one of the $T_i$ in Figure~\ref{fig:cycle-tangles} must be trivial or a part of a twist region.  If trivial, then the square $S_i$ is not normal, which is a contradiction.  But because $K$ has a diagram with no bigons, $T_i$ cannot contain bigons, so $T_i$ must be a single crossing.  Note neither $T_{i-1}$ nor $T_{i+1}$ can be a single crossing (it could be that $T_{i-1}=T_{i+1}$), else $T_i$ and this tangle would form a bigon.  Thus $T_i \cup T_{i-1}$ is a tangle that is non-trivial, and does not bound a portion of a twist region.  If $T_{i+1}$ is a distinct tangle, then we have a cycle of tangles (possibly after performing this same move elsewhere to move single crossings into larger tangles).

The only remaining possibility is that there are just two tangles, $T_1$ and $T_2$, and $T_1$ is a single crossing, and $T_2$ is a knotty tangle.  But then $S_1$ and $S_2$ can both be isotoped to encircle the ideal vertex corresponding to the single crossing, as in Figure~\ref{fig:cycle-tangles} (right). Then $S_1 \cup S_2$ is a boundary parallel annulus in $M_B \cut DR$, parallel to the double of the ideal vertex corresponding to this single crossing.  A boundary parallel annulus is not essential.
\end{proof}

\begin{lemma}\label{lemma:step2-guts}
  Let $K_R$ be a hyperbolic link with prime, alternating, twist--reduced diagram with no red bigons and no cycle of tangles, with checkerboard surfaces $B_R$ and $R_R$, and let $E$ be the Seifert fibered solid tori from Lemma~\ref{lemma:step1-cut1}.  Finally, let $K_{BR}$ be the new link obtained by replacing any adjacent blue bigons from the diagram of $K_R$ with a single crossing, and let $B_{BR}$ and $R_{BR}$ be the checkerboard surfaces of $K_{BR}$. 
Then
\[
\guts(D((S^3\cut B_R)-E)\cut DR_R) = \guts( D(S^3\cut B_{BR})\cut DR_{BR}) = 
D(S^3\cut B_{BR})\cut DR_{BR}.
\]
Thus
\[
\vol(D((S^3\cut B_R)-E)) \: \geq \: \frac{1}{2}
\vol( D( D(S^3\cut B_{BR})\cut DR_{BR}) ).
\]
\end{lemma}

\begin{proof}
Because $K_{BR}$ has no bigons, Lemma~\ref{lemma:step2-guts-nobigons} implies the last equality:
\[
\guts(D(S^3\cut B_{BR})\cut DR_{BR}) \: = \: D(S^3\cut B_{BR})\cut DR_{BR}.
\]
Hence it remains to show the first equality.

Recall that $S^3\cut B_R$ is obtained by gluing two checkerboard ideal polyhedra along their red faces, leaving blue faces unglued.  The Seifert fibered solid tori of $E$ lie in those polyhedra in blue twist regions, meeting the bigons of the twist region and the adjacent red faces, as on the right of Figure~\ref{fig:fused-units}.

When we double along the blue surface to construct $D(S^3\cut B_R)$, the Seifert fibered solid tori in $E$ glue to give a Seifert fibered submanifold, with boundary in $D(S^3\cut B_R)$ an essential torus obtained by gluing two annuli $A$, where $A$ is made up of squares bounding the fused units in Figure~\ref{fig:fused-units}.  Thus
$(S^3\cut B_R)-E$ consists of portions of the polyhedra that lie outside of $E$.  For each such solid torus, in each polyhedron these consist of regions bounding a single bigon, and one region bounding the fused unit with the non-trivial tangle.

Note that each region consisting of a square encircling a single bigon is fibered.  The bigon itself is fibered, with fibers meeting each edge of the bigon in a single point and parallel to the ideal vertex, which is part of the parabolic locus.  Similarly, the square encircling the bigon gives a fibered disk.  Together, these two squares bound a fibered box, where fibers have one endpoint on one red face, one endpoint on the other, and are parallel to the fibers of the bigon.  

To obtain $D((S^3\cut B_R)-E)\cut DR_R$, we take four polyhedra, and glue them in pairs by the identity along their blue faces.  Notice that the fibered boxes at blue twist regions must belong to the characteristic $I$--bundle of the cut manifold, so these twist regions cannot be part of the guts of this manifold.  Therefore the guts is obtained by considering only the quadrilateral bounding the non-trivial tangle.  The corresponding polyhedron is equivalent to the polyhedron obtained by replacing the blue twist region with a single bigon.

But now consider any remaining blue bigons in the diagram, including this, and including blue bigons from twist regions that did not give rise to Seifert fibered solid tori in the previous step.  As before, a neighborhood of any such bigon and the parabolic locus is an $I$--bundle, and is part of the characteristic $I$--bundle of the manifold.  So if $Y$ is a neighborhood of the union of the parabolic locus and all the bigons in the polyhedra, then the guts of $D((S^3\cut B_R)-E)\cut DR_R :=N_R$ is the guts of the closure of $N_R - Y$, where the latter is given parabolic locus ${\rm{cl}}(\bdy Y - \bdy N_R)$.

As in the beginning of Section~5 of~\cite{lackenby}, this can be identified explicitly.  If we replace the bigons of each blue twist region by a single ideal vertex of the polyhedron, then the remaining portions of the polyhedra will be identical.  But this is exactly the polyhedron of the link with no bigons $K_{BR}$.  The desired result follows.  
\end{proof}

\begin{lemma}\label{lemma:step2-homeo}
When $K$ is a link with prime, alternating, twist--reduced diagram with no bigons and no cycle of tangles, the doubles of manifolds $M_B\cut DR$ and $M_R\cut DB$ admit isometric finite volume hyperbolic structures.  In these structures, the surfaces coming from $DR$ and $DB$ are totally geodesic and meet at angle $\pi/2$.  The manifolds are both obtained by gluing eight isometric copies of a right angled hyperbolic ideal polyhedron $P$, and this polyhedron is equivalent to the checkerboard ideal polyhedron of $K$.
\end{lemma}

\begin{proof}
By Lemma~\ref{lemma:step2-guts-nobigons}, $M_B\cut DR$ is all guts, with no $I$--bundle or Seifert fibered pieces.  Thus when we double it, the resulting manifold admits a complete finite volume hyperbolic structure, in which $DR$ is a totally geodesic surface.  The same argument applies to the double of $M_R\cut DB$.

Now we show that these two doubles are homeomorphic.  Both are homeomorphic to eight copies of the checkerboard polyhedron $P$ coming from $K$, glued by identity maps on faces, as follows.  In particular, recall that $M_B\cut DR = D(S^3\cut B)\cut DR$ is obtained by taking four copies of the polyhedron $P$, denoted by $P_1^1$, $P_1^2$, $P_2^1$, and $P_2^2$, with red faces unglued, and blue faces of $P_j^1$ glued to those of $P_j^2$ by the identity, for $j=1, 2$.  When we double across $DR$, we obtain four more polyhedra, $\bar{P}_j^i$, $i,j=1,2$, with blue faces of $\bar{P}_j^1$ glued to blue faces of $\bar{P}_j^2$ by the identity for each $i=1, 2$, and red faces of $\bar{P}_j^i$ glued to red faces of $P_j^i$ by the identity, for $i,j=1,2$.  
To form the double of $M_R\cut DB$, we repeat the process, only cut and glue along red first, then along blue.  Again we obtain eight copies of the checkerboard polyhedron, denoted $Q_i^j$ and $\bar{Q}_i^j$, for $i,j =1, 2$, with $Q_i^1$ glued to $Q_i^2$ by the identity map on red faces, $\bar{Q}_i^1$ glued to $\bar{Q}_i^2$ by the identity on red faces, and $\bar{Q}_i^j$ glued to $Q_i^j$ by the identity on blue faces.  This gluing is summarized in the following diagram.

\[\xymatrixcolsep{5pc}\xymatrix{
  {P_i^1}\ar@{<->}[d]_{\mbox{red}}\ar@{<->}[r]^{\mbox{blue}}
  & {P_i^2}\ar@{<->}[d]^{\mbox{red}}
  & {Q_i^1}\ar@{<->}[d]_{\mbox{blue}}\ar@{<->}[r]^{\mbox{red}}
  & {Q_i^2}\ar@{<->}[d]^{\mbox{blue}}
  \\
  {\bar{P}_i^1}\ar@{<->}[r]^{\mbox{blue}} 
        & {\bar{P}_i^2}
  & {\bar{Q}_i^1}\ar@{<->}[r]^{\mbox{red}} 
        & {\bar{Q}_i^2}
}\]

Now build a homeomorphism by mapping by the identity between $P$'s and $Q$'s, rotating the diagram on the left to match that on the right.  That is, map $P_i^1$ to $Q_i^2$, map $\bar{P}_i^1$ to $Q_i^1$, map $P_i^2$ to $\bar{Q}_i^2$, and map $\bar{P}_i^2$ to $\bar{Q}_i^1$, for $i=1,2$.  These maps give the identity on the interiors of the polyhedra, the identity on interiors of faces, and extend to identity maps on edges between faces.  Thus they give a homeomorphism of spaces.  

Since $D(M_R\cut DB)$ and $D(M_B\cut DR)$ are finite volume hyperbolic manifolds, by Mostow--Prasad rigidity \cite{mostow, prasad}, the doubles are actually isometric.  Thus $DB$ and $DR$ are totally geodesic in each of them.  Hence cutting along these totally geodesic surfaces yields eight copies of $P$, each with totally geodesic red and blue faces.

Finally, since $DR$ is geodesic when we double along $DB$, it must intersect $DB$ at right angles.  Thus the surfaces meet everywhere at angle $\pi/2$, as claimed.  
\end{proof}

\begin{theorem}\label{thm:volbound}
Suppose $K$ is a link with a prime, alternating, twist--reduced, diagram with no cycle of tangles.  
Let $K'$ be the link with diagram obtained from that of $K$ by replacing any adjacent red bigons with a single crossing, and by replacing any adjacent blue bigons with a single crossing. 
Let $P(K')$ denote the checkerboard polyhedron coming from $K'$, given an ideal hyperbolic structure with all right angles.  Then
\[
\vol(S^3-K) \: \geq \: 2 \, \vol(P(K')).
\]
\end{theorem}

\begin{proof}
Lemmas~\ref{lemma:step1-cut1}, \ref{lemma:step2-cut2}, and~\ref{lemma:step2-guts} imply that when $K$ admits a prime, alternating diagram with no bigons and no cycle of tangles, 
\[
\vol(S^3-K) \: \geq \: \frac{1}{4}\vol(D(M_B\cut DR)).
\]
Lemma~\ref{lemma:step2-homeo} further implies that $\vol(D(M_B\cut DR)) = 8\,\vol(P(K))$.
Hence $\vol(S^3-K) \geq 2\,\vol(P(K))$.

If $K$ contains bigons, let $K_B$ and $K_R$ denote the link with the blue and red bigons removed, respectively, and let $K_{BR} = K'$ denote the link with both blue and red bigons removed. Let $B_B$ and $R_B$ denote the checkerboard surfaces of $K_B$, let $B_R$ and $R_R$ denote the checkerboard surfaces of $K_R$, and let $B_{BR}$ and $R_{BR}$ denote the checkerboard surfaces of $K_{BR}$.

Lemma~\ref{lemma:step1-cut1} implies that
\[ \vol(S^3-K) \: \geq \: \frac{1}{2}\vol(D((S^3\cut B_R)-E_R)), \]
where $E_R$ is a collection of Seifert fibered solid tori.
By Lemma~\ref{lemma:step2-guts},
\[
\vol(D((S^3\cut B_R)-E_R)) \: \geq \: \frac{1}{2}\vol(D(D(S^3\cut B_{BR})\cut DR_{BR})).
\]
By Lemma~\ref{lemma:step2-homeo}, $\vol(D(D(S^3\cut B_{BR})\cut DR_{BR})) = 8\,\vol(P(K_{BR}))$.
Thus
\[
\vol(S^3-K) \: \geq \: 2\,\vol(P(K_{BR})).
\]
\end{proof}

\section{Geometrically maximal knots}\label{sec:vollims}

In this section, we complete the proof of Theorem~\ref{thm:geommax}.

Since the number of crossings in a diagram of $K$ is equal to the
number of ideal vertices of $P(K)$ by item~\eqref{item:poly1} of
Definition~\ref{def:polyhedra}, our goal is to bound the ratio of the
volume of $P(K)$ to the number of vertices of $P(K)$.  We will do so
using methods of Atkinson~\cite{atkinson:volume}, which rely on
fundamental results of He~\cite{he:rigidity} on the rigidity of circle
patterns.  In particular, we employ the proof of Proposition~6.3 of
\cite{atkinson:volume}, which obtains the volume per vertex bounds we
desire, but for a different class of polyhedra. We first set up
notation.

If we lift the ideal polyhedron $P:=P(K)$ into $\HH^3$, the geodesic faces lift to lie on geodesic planes.  These correspond to Euclidean hemispheres, and each extends to give a circle on $\CC$.  For every such polyhedron, we obtain a finite collection of circles (or disks) on $\CC$ meeting at right angles in pairs, and meeting at ideal vertices in sets of four.  This defines a finite \emph{disk pattern} $D$ on $\CC$, with angle $\pi/2$ between disks.  Let $G(D)$ be the graph with a vertex for each disk and an edge between two vertices when the corresponding disks overlap on $\CC$.  Edges of $G(D)$ are labeled by the angle at which the disks meet, which in our case is $\pi/2$ for each edge.  Note all faces of $G(D)$ in our case are quadrilaterals, since all vertices of $P$ are 4--valent, hence the disk pattern $D$ is \emph{rigid} \cite{he:rigidity}.

Similarly, as described in Section~\ref{sec:weave}, we can form an infinite polyhedron $P_{\W}$ corresponding to a checkerboard polyhedron of the infinite weave $\W$:  view the diagram of $\W$ as squares with vertices on the integer lattice, and for each square draw the Euclidean circle on $\CC$ running through its four vertices.  Each such circle on $\CC$ corresponds to a hemisphere in $\HH^3$.  Let $P_{\W}$ be the infinite polyhedron obtained from $\HH^3$ by cutting out all half--spaces of $\HH^3$ bounded by these hemispheres.  (In Section~\ref{sec:weave}, this polyhedron was called $\tilde{X_1}$.)  By construction, faces meet in pairs at right angles, and at ideal vertices in fours.  We obtain a corresponding rigid disk pattern $D_\infty$, as in Figure~\ref{fig:W_top}(a).

\begin{define}\label{def:agreetogeneration}
Let $D$ and $D'$ be disk patterns.  Give $G(D)$ and $G(D')$ the path metric in which each edge has length $1$.  For disks $d$ in $D$ and $d'$ in $D'$, we say $(D,d)$ and $(D',d')$ \emph{agree to generation $n$} if the balls of radius $n$ about vertices corresponding to $d$ and $d'$ admit a graph isomorphism, with labels on edges preserved.
\end{define}

For any disk $d$, we let $S(d)$ be the geodesic hyperplane in $\HH^3$ whose boundary agrees with that of $d$.  That is, $S(d)$ is the Euclidean hemisphere in $\HH^3$ with boundary on the boundary of $d$.  For a disk pattern coming from a right angled ideal polyhedron, the planes $S(d)$ form the boundary faces of the polyhedron. In this case, the disk pattern $D$ is said to be \emph{simply connected}, meaning the union of the disks form a simply connected region, and \emph{ideal}, since it corresponds to an ideal polyhedron.

If $d$ is a disk in a disk pattern $D$, with intersecting neighboring disks $d_1, \dots, d_m$ in $D$, then $S(d) \cap S(d_i)$ is a geodesic $\gamma_i$ in $\HH^3$.  Assume that the boundary of $d$ is disjoint from the point at infinity.  Then the geodesics $\gamma_i$ on $S(d)$ bound an ideal polygon in $\HH^3$, and we may take the cone over this polygon to the point at infinity.  Denote the ideal polyhedron obtained in this manner by $C(d)$ (see Figure~\ref{fig:disk-pattern}).

\begin{figure}
\begin{tabular}{ccc}
 \includegraphics[scale=1.1]{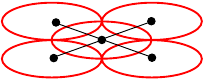} & \qquad &
 \includegraphics[scale=1.1]{figures/octahedron_circles2} \\
(a) & \qquad & (b) \\
\end{tabular}
\caption{(a) Disk pattern $(D,d)$ and graph $G(D)$.
  (b) Ideal polyhedron $C(d)$.}
  \label{fig:disk-pattern}
\end{figure}

The following lemma restates Lemma~6.2 of \cite{atkinson:volume}.

\begin{lemma}[Atkinson \cite{atkinson:volume}]\label{lemma:atkinson}
There exists a bounded sequence $0 \leq \epsilon_{\ell} \leq b <\infty$ converging to zero such that if $D$ is a simply connected, ideal, rigid, finite disk pattern containing a disk $d$ so that $(D_\infty, d_\infty)$ and $(D, d)$ agree to generation $\ell$ then
\[
|\vol(C(d)) - \vol(C(d_\infty))| \: \leq \: \epsilon_{\ell}.
\]
\end{lemma}

We now generalize Proposition~6.3 of \cite{atkinson:volume} to classes of polyhedra that include the checkerboard ideal polyhedra of interest in this paper.  The proof of the following lemma is essentially contained in \cite{atkinson:volume}, but we present it here for completeness.

\begin{lemma}\label{lemma:atkinsonpoly}
Let $D_\infty$ denote the infinite disk pattern coming from $P_{\W}$, as defined above, with fixed disk $d_\infty$.  Let $P_n$ be a sequence of right angled hyperbolic polyhedra with corresponding disk patterns $D_n$.  Suppose the following hold.
\begin{enumerate}
\item If $F_\ell^n$ is the set of disks $d$ in $D_n$ such that $(D_n, d)$ agrees to generation $\ell$ but not to generation $\ell+1$ with $(D_\infty, d_\infty)$, then
\[
\lim_{n\to\infty} \frac{\left| \bigcup_\ell F_\ell^n \right|} {v(P_n)} = 1,
\mbox{ and }
\lim_{n\to\infty} \frac{|F^n_\ell|}{ v(P_n) } = 0.
\]
\item For every positive integer $k$, let $|f^n_k|$ denote the number of faces of $P_n$ with $k$ sides that are not contained in $\cup_\ell F^n_\ell$ and do not meet the point at infinity.  Then
\[
\lim_{n\to\infty} \frac{ \sum_k k|f^n_k| }{ v(P_n) } = 0.
\]
\end{enumerate}
Under these hypotheses,
\[
\lim_{n\to \infty} \frac{\vol(P_n)}{v(P_n)} = \frac{\voct}{2}.
\]
\end{lemma}

\begin{proof}
First, let $f^n_k$ be a face with $k$ sides that is not contained in $\cup_\ell F^n_\ell$, and which does not meet the point at infinity.  Then $\vol(C(f^n_k))$ has volume at most $k$ times the maximum value of the Lobachevsky function $\Lambda$, which is $\Lambda(\pi/6)$ \cite[Chapter~7]{thurston:notes}.  Let $E^n$ denote the sum of the actual volumes of all the cones over the faces $f^n_k$, for every integer $k$.  Then we have
\[
E^n \leq \sum_k\sum_{f^n_k} k\Lambda(\pi/6) = \sum_k k\,|f^n_k|\,\Lambda(\pi/6).
\]

For any face $f$ in $F^n_\ell$, let $\delta^n_\ell(f)$ be a positive number such that $\vol(C(f)) = \voct/2 \pm \delta^n_\ell(f)$.
Then
\[
\vol(P_n) = \sum_\ell \sum_{f\in F^n_\ell} \left(\frac{\voct}{2} \pm \delta^n_\ell(f)\right) + E^n.
\]
Hence
\[
\vol(P_n) = \left|\bigcup_\ell F^n_\ell \right|\frac{\voct}{2} + \sum_\ell\sum_{f\in F^n_\ell}(\pm \delta^n_\ell(f)) + E^n.
\]

We divide each term by $v(P_n)$ and take the limit.  For the first term, we obtain
\[
\lim_{n\to \infty} \frac{\left|\bigcup_\ell F^n_\ell\right|}{v(P_n)}\,\frac{\voct}{2} = \frac{\voct}{2}.
\]
By Lemma~\ref{lemma:atkinson}, there are positive numbers $\epsilon_\ell$ such that $\delta^n_\ell(f) \leq \epsilon_\ell$, so the second term becomes
\[
\lim_{n\to\infty} \frac{|\sum_\ell\sum_{f\in F^n_\ell} (\pm\delta^n_\ell(f))|}{v(P_n)} \leq \lim_{n\to\infty} \frac{\sum_\ell |F^n_\ell|\epsilon_\ell}{v(P_n)}.
\]
This can be seen to be zero, as follows.  Fix any $\varepsilon>0$.
Because $\lim\limits_{\ell\to\infty}\epsilon_\ell=0$, there is $L$ sufficiently large that $\epsilon_\ell < \varepsilon/3$, for $\ell>L$.  Then $\sum_{\ell=1}^{L} \epsilon_\ell$ is a finite number, say $M$.  By item (1) in the statement of this lemma, there exists $N$ such that if $n>N$ then $\max_{\ell\leq L}|F_{\ell}^n|/v(P_n) < \varepsilon/(3M\cdot L)$ and $|\cup_\ell F_\ell^n|/v(P_n) < (1+\varepsilon)$.
Then for $n> N$,
\[
\frac{\sum_\ell|F_\ell^n|\epsilon_\ell}{v(P_n)} =
\frac{\sum_{\ell=1}^{L} |F_\ell^n|\epsilon_\ell}{v(P_n)} + \frac{\sum_{\ell>L} |F_\ell^n|\epsilon_\ell}{v(P_n)} < \frac{\varepsilon L}{3M\cdot L}\,M + (1+\varepsilon)\frac{\varepsilon}{3} < \varepsilon.
\]
Hence the limit of the second term is zero.

Finally, the third term gives us
\[
\lim_{n\to\infty} \frac{E^n}{v(P_n)} \leq
\frac{\sum_k k\,|f_k^n|\,\Lambda(\pi/6)}{v(P_n)} = 0.
\]

Therefore, $\lim\limits_{n\to\infty} \vol(P_n)/v(P_n) = \voct/2.$
\end{proof}

\pagebreak

We can now prove Theorem ~\ref{thm:geommax}, which we recall from the introduction.

\begin{named}{Theorem~\ref{thm:geommax}}
Let $K_n$ be any sequence of hyperbolic alternating link diagrams that contain no cycle of tangles, such that
\begin{enumerate}
\item there are subgraphs $G_n\subset G(K_n)$ that form a F\o lner sequence for $G(\W)$, and
\item $\lim\limits_{n\to\infty} |G_n|/ c(K_n) = 1$.
\end{enumerate}
Then $K_n$ is geometrically maximal:
$\displaystyle \lim_{n\to\infty}\frac{\vol(K_n)}{c(K_n)}=\voct. $
\end{named}

\begin{proof}
We may assume that $K_n$ are prime and twist--reduced diagrams.  Any hyperbolic link is prime.  If a sequence $K_n$ satisfies the other conditions above, then the twist--reduced diagrams $\tilde K_n$ also satisfy these conditions because twist reduction does not change crossing number, and only changes the diagram in $G(K_n)-G_n$.

We first consider the case that $K_n$ contains no bigons.
Because $K_n$ is prime, alternating, has no bigons and no cycle of tangles, Theorem~\ref{thm:volbound} implies that
$\vol(S^3-K_n) \geq 2\,\vol(P(K_n))$.

The crossing number of $K_n$ is equal to the number of vertices of $P(K_n)$, so dividing by crossing number gives
\[ \frac{\vol(S^3-K_n)}{c(K_n)} \: \geq \: 2\,\frac{\vol(P(K_n))}{v(P(K_n))}.\]

The disk pattern graph $G(D_n)$ associated with $P(K_n)$ is the dual
of $G(K_n)$. Since $G(\W)$ is isomorphic to its own dual, to simplify
notation we may assume $G_n\subset G(D_n)$.  We will also assume that
the polygon enclosed by $G_n$ is simply connected, otherwise $G_n$ can
be modified by removing cutsets without affecting the required limits.
Pick a point in $G(D_n)$ which is outside $G_n$ to send to
infinity. We want to apply Lemma~\ref{lemma:atkinsonpoly}.

For condition (2) of Lemma~\ref{lemma:atkinsonpoly}, by counting vertices we obtain $\sum_k k\,|f_k^n|\leq 4\,|G(K_n)-G_n|$.  
The factor $4$ appears because every vertex belongs to four faces, so it will be counted at most four times in the sum.
Condition (2) now follows from $|G(K_n)|=v(P(K_n))=c(K_n)$ and $\displaystyle\lim_{n\to\infty} \frac{|G_n|}{c(K_n)} = 1$.

\begin{figure}
\includegraphics[width=4in]{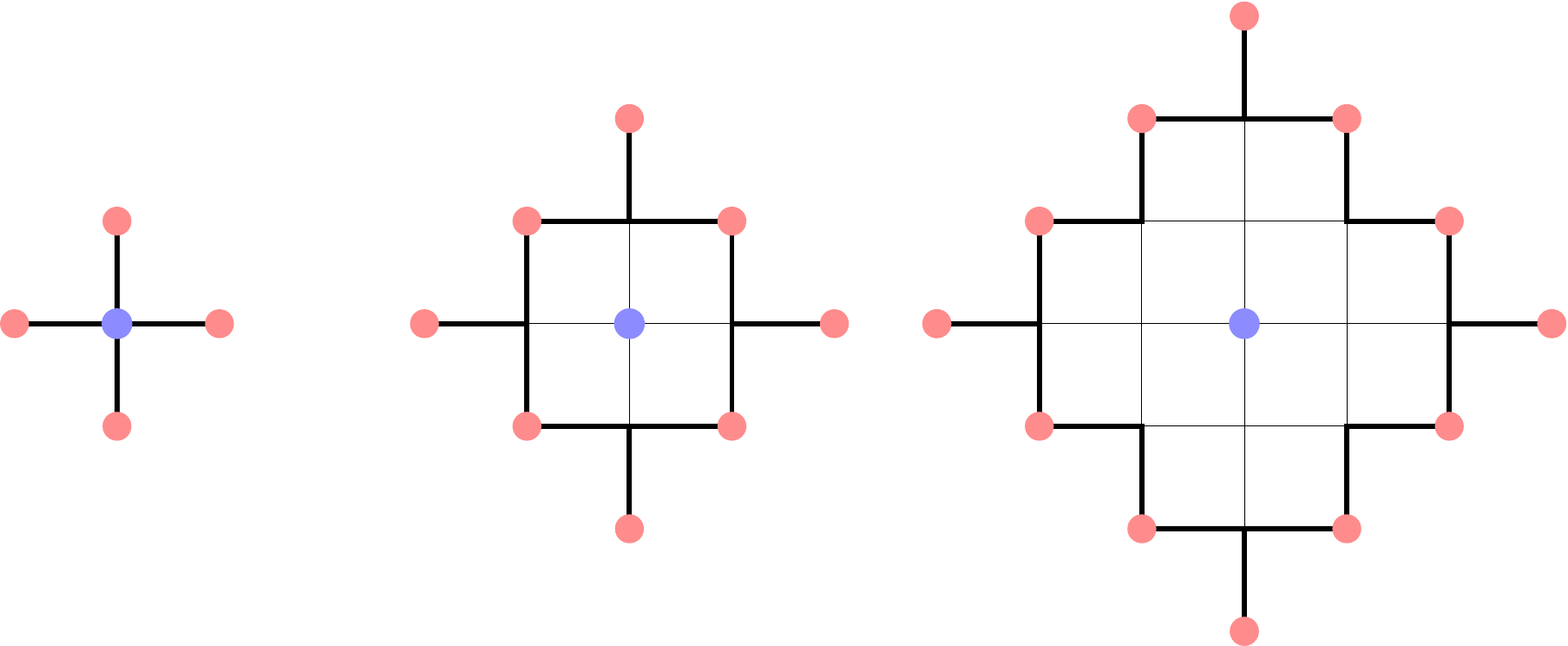}
\caption{Balls $B(v,r)$ for $r= 1,\, 2,\, 3$ in the square
  lattice. The vertices of $\partial B(v,r)$ are shown in bold.   }
\label{sq-lat-nbd}
\end{figure}

We now show that condition (1) of Lemma~\ref{lemma:atkinsonpoly} is
also satisfied.  
The sets $F_\ell^n$ consist of disks $d$ such that $(D_n, d)$ agrees
to generation $\ell$ but not to generation $\ell+1$ with
$(D_\infty, d_\infty)$.
Let $d\in F_\ell^n$ and let $v$ be the vertex in $G_n$ corresponding to $d$. 
Let $B(v,\ell) \subset G_n$ denote the ball centered at $v$ of radius $\ell$ in the path metric on $G_n$.
Then $d\in F_\ell^n$ means that $B(v,\ell) \subset G_n$ but
$B(v,\ell+1) \not\subset G_n$.  Hence, 
the distance from $v$ to $\partial G_n$ equals $\ell$, 
so that
$v \in \partial B(x,\ell)$ for some $x \in \partial G_n$. Thus,
$\displaystyle F_\ell^n \subset \cup_{x\in \partial G_n} \partial B(x,\ell)$.  

Since
$G_n \subset G(\W)$ which is the square lattice,
$|\partial B(x,\ell)|=4 \ell$ for any vertex $x \in G(\W)$.  See Figure
\ref{sq-lat-nbd}. Hence, $|F_\ell^n|\leq 4 \ell \, |\partial G_n|$.
Thus, we obtain one part of condition (1):
\[ \lim_{n\to\infty} \frac{|F_\ell^n|}{v(P(K_n))} \leq \lim_{n\to\infty}\frac{4 \ell \, 
|\partial G_n|}{c(K_n)} = 4 \ell\, \lim_{n\to\infty}\frac{ |\partial G_n|}{|G_n|}\cdot
\frac{|G_n|}{c(K_n)} = 0. \]

The fact that $G_n \subset G(\W)$ also implies that every vertex in
$G_n-\partial G_n$ corresponds to a disk in $F_\ell^n$ for some
$\ell$, and no vertex in $\partial G_n$ corresponds to a disk in any
of the $F_\ell^n$.  Hence, $|\cup_\ell F_\ell^n|= |G_n-\partial G_n|$.
Now, $\displaystyle\lim_{n\to\infty}\frac{|\partial G_n|}{|G_n|}=0$ and $\displaystyle\lim_{n\to\infty} \frac{|G_n|}{c(K_n)} = 1$ imply the second part of condition (1):
\[ \lim_{n\to\infty}\frac{|\cup_\ell F_\ell^n|}{v(P(K_n))}=\lim_{n\to\infty}\frac{|G_n-\partial G_n|}{c(K_n)}=1. \]

Thus by Lemma~\ref{lemma:atkinsonpoly},
\[ \lim_{n\to\infty} \frac{\vol(S^3-K_n)}{c(K_n)} \: \geq \: \lim_{n\to\infty} 2\,\frac{\vol(P(K_n))}{v(P(K_n))} \: = \: \voct. \]
Since equation (\ref{eq:dylan}) tells us that $\vol(S^3-K_n)/c(K_n) \leq \voct$, the above limit must equal $\voct$.

Finally, if $K_n$ contains bigons, let $K'_n$ denote the link with
both blue and red bigons removed.  Theorem~\ref{thm:volbound} implies
that $\vol(S^3-K_n) \geq 2\,\vol(P(K'_n))$.

Now, since all bigons of $K_n$ must be in $G(K_n)-G_n$, the above proof implies
\[
\lim_{n\to\infty}\frac{\vol(S^3-K_n)}{c(K_n)} \: \geq \: \lim_{n\to\infty}2\,\frac{\vol(P(K'_n))}{v(P(K'_n))}= \voct.
\]
Again by equation (\ref{eq:dylan}), this limit must also equal $\voct$. 
\end{proof}

\bibliographystyle{amsplain}
\bibliography{references}

\end{document}